%% file: main_arxiv.tex
\documentclass[final,onefignum,onetabnum,letterpaper]{siamonline190516}

\input{main_shared}

\ifpdf
\hypersetup{
  pdftitle={Glaubitz2021SBP_RBF},
  pdfauthor={Philipp \"Offner}
}
\fi

\begin{document}

\maketitle

\begin{abstract}
	\input{0_abstract}

\end{abstract}

\begin{keywords}
	Global radial basis functions, 
	time-dependent partial differential equations, 
	energy stability, \\
	summation-by-part operators 
\end{keywords}

\begin{AMS} 
      65N35, 
      65N12, 
      65D12, 
     65D25 
\end{AMS}

\input{1_introduction}

\input{2_preliminaries}

\input{3_FSBP}

\input{4_RBF-SBP}

\input{5_SBP_RBF_Operators}

\input{num_test}

\input{summary}

\appendix 
\input{app_polynomial}
\input{app_nondiagonal}\small
\bibliographystyle{siamplain}
\bibliography{literature}

\end{document}

%% file: main_shared.tex
\usepackage{soul}
\usepackage{lipsum}
\usepackage{amsfonts}
\usepackage{epstopdf}
\ifpdf
  \DeclareGraphicsExtensions{.eps,.pdf,.png,.jpg}
\else
  \DeclareGraphicsExtensions{.eps}
\fi

\usepackage{datetime}
\newdateformat{monthyeardate}{%
  \monthname[\THEMONTH] \THEDAY, \THEYEAR}

\usepackage{academicons}
\usepackage{xcolor}
\renewcommand{\orcid}[1]{\href{https://orcid.org/#1}{\textcolor[HTML]{A6CE39}{orcid.org/#1}}}

\usepackage{amsmath}
\allowdisplaybreaks
\usepackage{amssymb}
\usepackage{commath}
\usepackage{mathtools}
\usepackage{bbm}

\usepackage{color}
\usepackage{graphicx}
\usepackage[small]{caption}
\usepackage{subcaption}

\usepackage{relsize}
\usepackage{adjustbox}
\usepackage{algorithm}
\usepackage[noend]{algpseudocode}
\usepackage{booktabs}
\usepackage{tikz}

\usepackage[normalem]{ulem}
\usepackage{verbatim}

\usepackage{enumitem}
\setlist[enumerate]{leftmargin=.5in}
\setlist[itemize]{leftmargin=.5in}

\newsiamthm{problem}{Problem}
\newsiamremark{remark}{Remark}
\newsiamremark{hypothesis}{Hypothesis} 
\crefname{hypothesis}{Hypothesis}{Hypotheses}
\newsiamremark{example}{Example}
\newsiamthm{claim}{Claim}
\newsiamthm{conjecture}{Conjecture}

\headers{Energy stable RBFs on SBP form}{J.\ Glaubitz, J.\ Nordstr\"om, and P.\ \"Offner}

\title{Energy-stable global radial basis function methods on summation-by-parts form
\thanks{
\monthyeardate\today 
\corresponding{Philipp \"Offner} 
\funding{
JG was supported by AFOSR \#F9550-18-1-0316 and ONR MURI \#N00014-20-1-2595. 
JN was supported by Vetenskapsr\r{a}det, Sweden grant 2018-05084 VR and 2021-05484 VR, and the Swedish e-Science Research Center (SeRC). 
P\"O was supported by the Gutenberg Research College, JGU Mainz.
}}}

\author{
Jan Glaubitz\thanks{Department of Mathematics, Dartmouth College, Hanover, NH 03755, USA (\email{Jan.Glaubitz@Dartmouth.edu}, \orcid{0000-0002-3434-5563})}
\and 
{Jan Nordstr\"om\thanks{Department of Mathematics, Link\"oping University, 58183, Link\"oping, Sweden (\email{jan.nordstrom@liu.se}, \orcid{0000-0002-7972-6183})}} \thanks{Department of Mathematics and Applied Mathematics, University of Johannesburg, P.\,O.\ Box 524, Auckland Park 2006, Johannesburg, South Africa}
\and 
Philipp \"Offner\thanks{Institute of Mathematics, Johannes Gutenberg University, Mainz, Germany, (\email{poeffner@uni-mainz.de}, \orcid{0000-0002-1367-1917})} 
}

\usepackage{amsopn}

\newenvironment{eq}{\begin{equation}}{\end{equation}} 

\DeclareMathOperator{\diag}{diag}

\newcommand{\Span}{\mathrm{span}}

\renewcommand{\dim}{\mathrm{dim} \,}
\renewcommand{\d}{\mathrm{d}}
\newcommand{\intd}{\, \mathrm{d}}
\newcommand{\N}{\mathbb{N}}

\newcommand{\R}{\mathbb{R}}


%% file: 0_abstract.tex
Radial basis function methods are powerful tools in numerical analysis and have demonstrated good properties in many different simulations. 
However, for time-dependent partial differential equations, only a few stability results are known.
In  particular, if boundary conditions are included, stability issues frequently occur. 
The question we address in this paper is how provable stability for RBF methods can be obtained. 
We develop and construct energy-stable radial basis function methods using the general framework of summation-by-parts operators often used in the Finite Difference and Finite Element communities.

%% file: 1_introduction.tex
\section{Introduction} 
\label{sec:introduction} 

We investigate energy stability of global radial basis function (RBF) methods for time-dependent partial differential equations (PDEs). 
Unlike finite differences (FD) or finite element (FE) methods, RBF schemes are mesh-free, making them flexible with respect to the geometry of the computational domain since the only used geometrical property is the pairwise distance between two centers. 
Further, they are suitable for problems with scattered data like in climate \cite{flyer2012guide,lazzaro2002radial} or stock market \cite{cuomo2020greeks, pettersson2008improved} simulations.
Finally, for smooth solutions, one can reach spectral convergence \cite{flyer2016role, fornberg2005accuracy}. 
In addition, they have recently become more and more popular for solving time-dependent problems in quantum mechanics, fluid dynamics, etc.  \cite{dehghan2017numerical,hesthaven2020rbf,  iske2020ten, tominec2021residual}. 
One distinguishes between global RBF methods (Kansa's methods) \cite{kansa1990multiquadrics} and local RBF methods, such as the RBF generated finite difference (RBF-FD) \cite{tolstykh2000using} and RBF partition of unity (RBF-PUM) \cite{wendland2002fast} method. 
See the monograph \cite{fornberg2015primer} and references therein. \\
Even though their efficiency and good performance have been demonstrated for various problems, only a few stability results are known for advection-dominated problems. 
For example, an eigenvalue analysis was performed for a linear advection equation in \cite{platte2006eigenvalue}, and it was found that RBF discretizations often produced  eigenvalues lending to an exponential increase of the $L_2$ norm   when boundary conditions were introduced. 
To  illustrate this, consider the following example (also found in \cite[Section 6.1]{glaubitz2021stabilizing}): 
\begin{equation}\label{eq_simple_example}
	\partial_t u+ \partial_x u=0, \qquad u(x,0)= \mathrm{e}^{-20x^2} 
\end{equation}
with $x \in [-1,1]$, $t>0$, and where periodic boundary conditions are applied. 
In this example, a bump is traveling to the right, leaving the domain and coming back to the left. 
\begin{figure}[tb]
	\centering 
	\begin{subfigure}[b]{0.35\textwidth}
		\includegraphics[width=\textwidth]{%
      		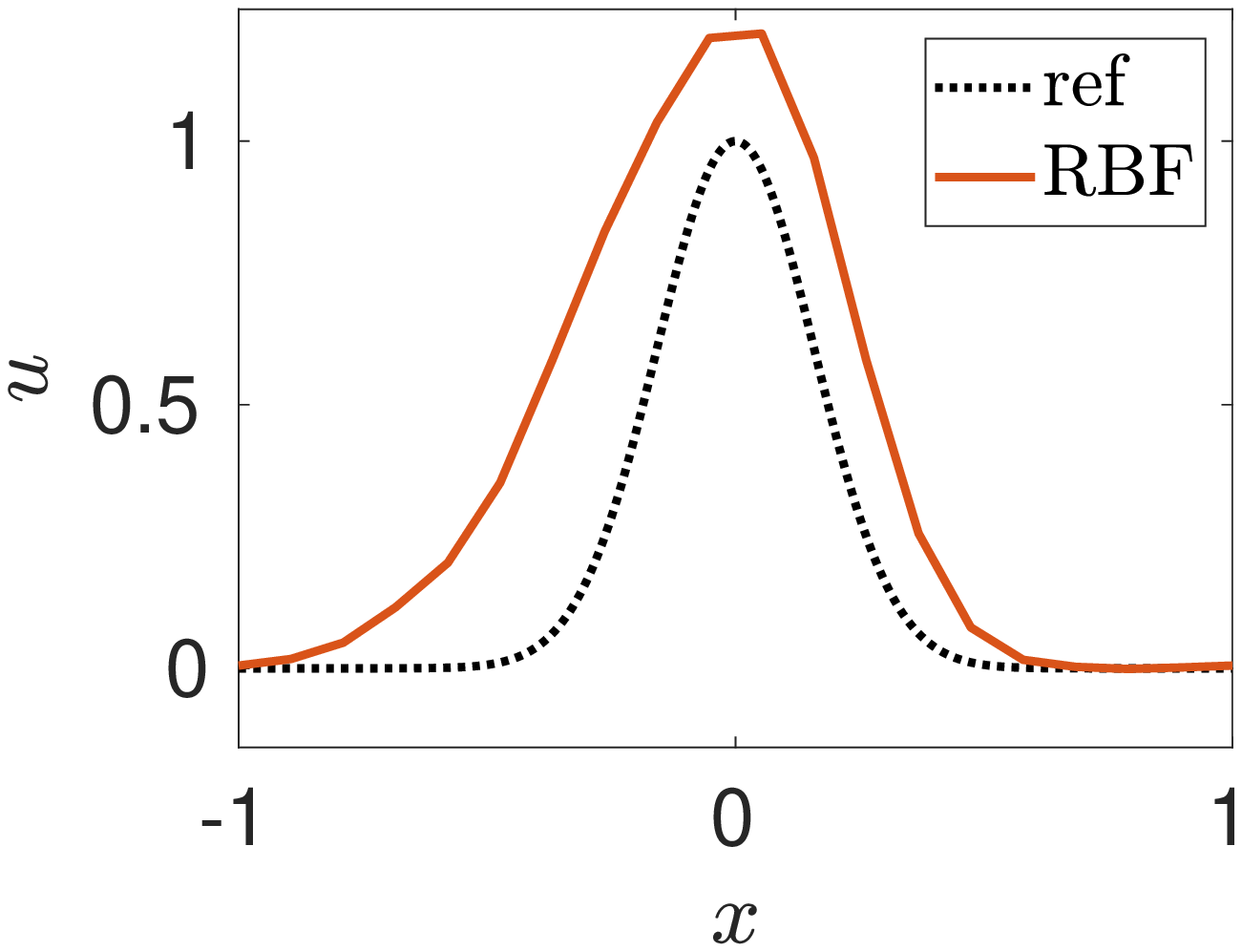} 
    		\caption{Numerical solution at $t=10$}
    		\label{fig:init_solution}
  	\end{subfigure}%
	~
  	\begin{subfigure}[b]{0.35\textwidth}
		\includegraphics[width=\textwidth]{%
      		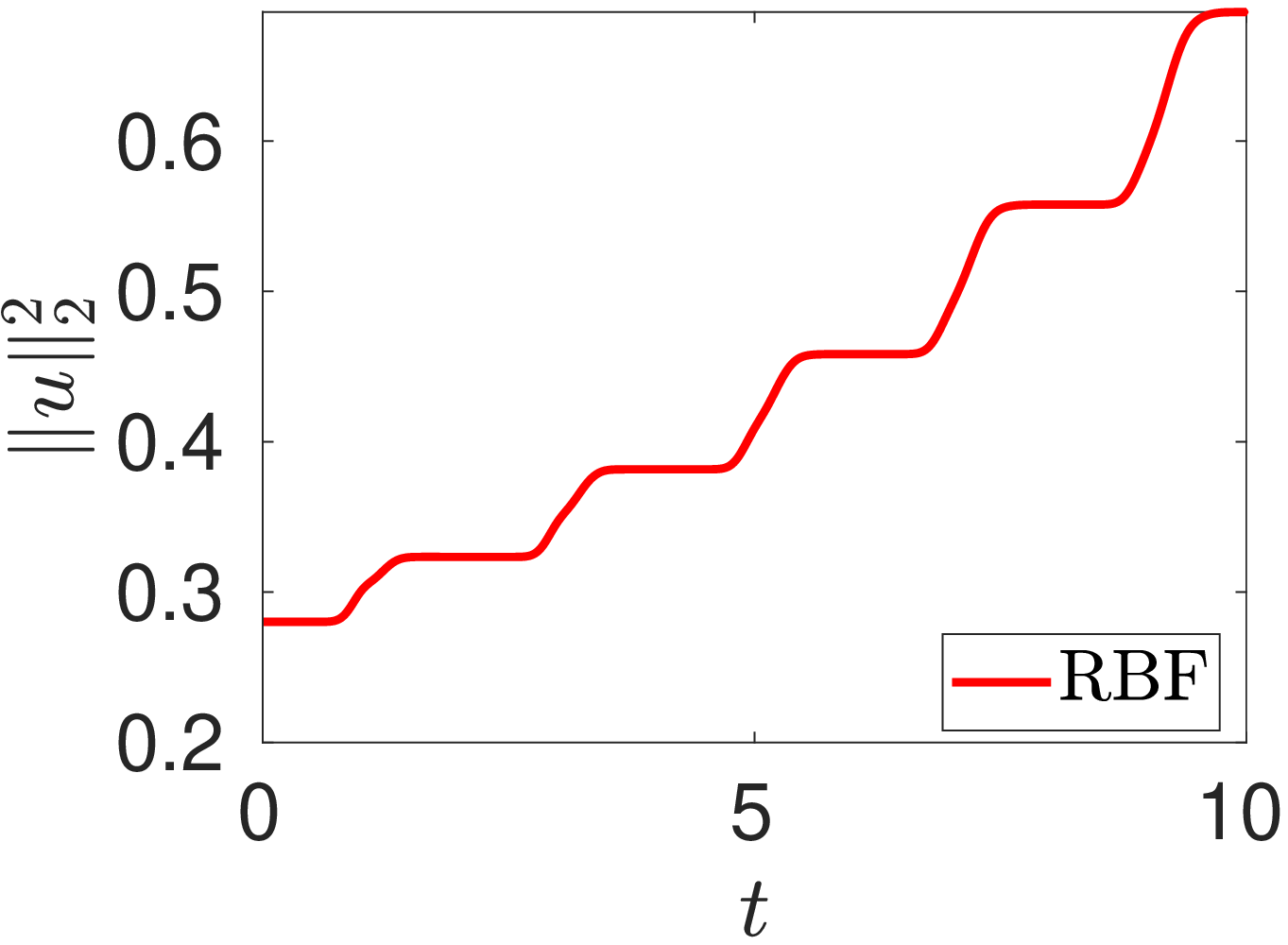} 
    		\caption{Energy profile developing in time}
    		\label{fig:init_energy}
  	\end{subfigure}%
  	\caption{
  	Gaussian kernel with $N=20$ points (equidistant points) after 10 periods 
  	}
  	\label{fig:init_approx}
\end{figure}  
In Figure \ref{fig:init_approx}, we plot the numerical solution and its energy  up to $t=10$ using a global RBF method with a Gaussian kernel and $N=20$ points.
An increase of the size of the bump and of the $L_2$ energy can be seen. For longer times, the computation breaks down. 
The discrete setting does not reflect the continuous one with zero energy growth and demonstrates the stability problems. \\
To overcome those, it was shown in \cite{glaubitz2021stabilizing,glaubitz2021towards}
that a weak formulation could result in a stable method. 
Recently, $L_2$ estimates were obtained using an oversampling technique \cite{tominec2021stability}. 
Both these efforts use special techniques, and the question we address in this paper is  how to stabilize RBF methods in a general way.\\
Classical summation-by-parts (SBP) operators were introduced during the 1970s in the context of FD schemes and they allow for a systematic development of energy-stable semi-discretizations of well-posed initial-boundary-value problems (IBVPs) \cite{fernandez2014review,svard2014review}. The SBP property is a  discrete analog to integration by parts, and proofs from the continuous setting carry over directly  to the discrete framework \cite{nordstrom2017roadmap} if proper boundary procedures are added \cite{svard2014review}. 
First based on polynomial approximations, the SBP theory has recently  been extended  to general function spaces developing so-called FSBP operators in  \cite{glaubitz2022nonpolnyomial}. 
Here, we investigate stability of global RBF methods 
through the lens of the FSBP theory. \\
We demonstrate that many existing RBF discretizations do not satisfy the FSBP property, which opens up for   instabilities in these methods.
Based on these findings,  we show how RBF discretizations can be modified to obtain an SBP property. 
This then allows for a systematic development of energy-stable RBF methods. 
We give a couple of concrete examples including the most frequently used RBFs, 
where $L_2$ estimates are derived using an oversampling technique. 
For simplicity, we focus on the univariate setting for developing an SBP theory in the context of global RBF methods. 
That said, RBF methods and SBP operators can easily be extended to the multivariate setting, which is also demonstrated in our numerical tests. \\
The rest of this work is organized as follows. 
In \cref{sec:prelim}, we provide some preliminaries on energy-stability of IBVPs and global RBF methods. 
Next, the concept of FSBP operators is shortly revisited  in \cref{sec:FSBP}. We adapt the FSBP theory to RBF function spaces in 
\cref{sec:SBP_RBF}. Here, it is also demonstrated that many existing RBF methods do not satisfy the SBP property and how to construct  RBF operators in SBP form (RBFSBP). 
In \cref{sec_Operators}, we give a couple of concrete examples of RBFSBP operators resulting in energy-stable methods. Finally, we provide numerical tests in \cref{sec:num_tests} and concluding thoughts in \cref{sec:summary}.

%% file: 2_preliminaries.tex
\section{Preliminaries} 
\label{sec:prelim} 

We now provide a few preliminaries on IBVPs and RBF methods.

\subsection{Well-posedness and Energy Stability}
Following \cite{gustafsson1995time,nordstrom2017roadmap,  svard2014review}, we consider 
\begin{equation}\label{eq_IBVP}
\begin{aligned}
    \partial_t u &= \mathcal{L}(x,t,\partial_x) u+ \mathcal{F}, \quad && x_L < x  < x_R, \ t > 0,\\
    u(x,0) & = f(x), \quad && x_L \leq x \leq x_R, \\
    \mathcal{B}_0(t,\partial_x) u(x_L,t) &= g_{x_L}(t), \quad && t \geq 0, \\
    \mathcal{B}_1(t,\partial_x) u(x_R,t) &= g_{x_R}(t), \quad && t \geq 0,
\end{aligned}
\end{equation}
where $u$ is the  solution and  $\mathcal{L}$ is a differential operator with smooth coefficients. 
Further, $B_0$ and $B_1$ are operators defining the boundary conditions, $\mathcal{F}$ is a forcing function, $f$ is the initial data, and $g_{x_L}, g_{x_R}$ denote the boundary data. 
Examples of \cref{eq_IBVP} include the advection equation 
\begin{equation}\label{eq:advection}
    \partial_t u(x,t) + a \partial_x u(x,t)=0
\end{equation}
with constant $a \in \R$, the diffusion equation
\begin{equation}\label{eq:advection_basic}
    \partial_t u(x,t) =  \partial_{x} \left( \kappa \partial_x u(x,t) \right)
\end{equation}
with $\kappa \in \R$ depending on $x,t$, as well as combinations of \cref{eq:advection,eq:advection_basic}. 
Let us now formalize what we mean by the IBVP \cref{eq_IBVP} being well-posed.

\begin{definition}\label{def_well_posed}
The IBVP \cref{eq_IBVP} with $ \mathcal{F}= g_{x_L}=g_{x_R}=0$ is \textbf{well-posed}, if for every $f\in C^{\infty}$ that vanishes in a neighborhood of $x=x_L,x_R$, \cref{eq_IBVP} has a unique smooth solution $u$ that satisfies 
\begin{equation}\label{eq_energy}
\norm{u(\cdot,t)}_{L_2}  \leq C \mathrm{e}^{\alpha_C t}\norm{f}_{L_2},
\end{equation} 
where $C, \alpha_c$ are constants independent of $f$. 
Moreover, the IBVP \cref{eq_IBVP} is  \textbf{strongly well-posed}, if it is well-posed and  
\begin{equation}\label{eq_strong_energy}
\norm{u(\cdot,t)}^2_{L_2} \leq C(t) \left( \norm{f}^2_{L_2}+ \int_{0}^{t} \left( \norm{\mathcal{F}(\cdot, \tau) }^2_{L_2} +|g_{x_L}(\tau)|^2+|g_{x_R}(\tau)|^2 \right) d \tau \right),
\end{equation} 
holds, where the function $C(t)$ is bounded for finite $t$ and independent of $\mathcal{F}, g_{x_L}, g_{x_R}$, and $f$.
\end{definition}
Switching to the discrete framework, our numerical approximation $u^h$ of  \cref{eq_IBVP} should be constructed in  such a way that similar estimates to  \cref{eq_energy} and \cref{eq_strong_energy} are obtained.
We denote our grid quantity (a measure of the grid size) by $h$. 
In the context of RBF methods, $h$ denotes the maximum distance between two neighboring points. 
We henceforth denote by $\|\cdot\|_h$ a discrete version of the $L_2$-norm and $\norm{\cdot}_b$ represents a discrete boundary norm. 
Then, we define stability of the numerical solution as follows.

\begin{definition}
Let  $ \mathcal{F}= g_{x_L}=g_{x_R}=0$  and $f^h$ be an adequate projection of the initial data $f$ which vanishes at the boundaries. 
The approximation $u^h$ is \textbf{stable} if  
\begin{equation}\label{semi_stability_not}
\norm{u^h(t)}^2_h \leq C \mathrm{e}^{\alpha_d t}\norm{f^h}_h
\end{equation} 
holds for all sufficiently small $h$, where $C$ and $\alpha_d$ are constants independent of $f^h$. 
The approximated solution $u^h$ is called \textbf{strongly energy stable} if it is stable and 
\begin{equation}\label{semi_stability}
\norm{u^h(t)}^2_h \leq C(t) \left( \norm{f^h}^2_h+ \max\limits_{\tau \in [0,t]} \norm{\mathcal{F}( \tau) }_h^2 +\max\limits_{\tau  \in [0,t]} \norm{g_{x_L}(\tau)}^2_b+\max\limits_{\tau \in [0,t]} \norm{g_{x_R}(\tau)}^2_b \right)
\end{equation}
holds for all sufficiently small $h$.
The function $C(t)$ is bounded for  finite $t$ and independent of 
$\mathcal{F}, g_{x_L}, g_{x_R}$, and $f^h$. 
\end{definition}

\subsection{Discretization}\label{subsec_RBF}

To discretize the  IBVP \cref{eq_IBVP}, we apply the method of lines.
The space discretization is done using a global RBF method resulting in a system of ordinary differential equations (ODEs):  
\begin{equation}\label{eq:semidis-eq}
  \frac{\d}{\d t} \mathbf{u} = \operatorname{L}(\mathbf{u}).
\end{equation}
Here, $\mathbf{u}$  denotes the vector of coefficients and  $\operatorname{L}$ represents the spatial operator. 
We used the explicit strong stability preserving (SSP) Runge--Kutta (RK) method of third-order with three stages (SSPRK(3,3)) \cite{shu1988total} for all subsequent numerical tests.

\subsubsection{Radial Basis Function Interpolation} 
\label{sub:RBF-interpol}
RBFs are  powerful tools for  interpolation and  approximation 
\cite{wendland2004scattered,fasshauer2007meshfree,fornberg2015primer}.
In the context of the present work, we are especially interested in RBF interpolants.
Let $u: \R \supset \Omega \to \R$ be a scalar valued function and $X_K = \{x_1,\dots,x_K\}$ a set of interpolation points, referred to as \emph{centers}.  
The \emph{RBF interpolant} of $u$ is 
\begin{equation}\label{eq:RBF-interpol}
  u^h(x) 
    = \sum_{k=1}^K \alpha_k \varphi( |x - x_k| ) + \sum_{l=1}^m \beta_k p_l(x).
\end{equation} 
Here, $\varphi: \R_0^+ \to \R$ is the \emph{RBF} (also called \emph{kernel}) and $\{p_l\}_{l=1}^m$ is a basis for the space of polynomials up to degree $m-1$, denoted by $\mathbb{P}_{m-1}$.
Furthermore, the RBF interpolant \cref{eq:RBF-interpol} is uniquely determined by the conditions 
\begin{alignat}{2}
	u^h(x_k) 
    		& = u(x_k), \quad 
		&& k=1,\dots,K, \label{eq:interpol_cond} \\ 
  	\sum_{k=1}^K \alpha_k p_l(x_k) 
    		& = 0 , \quad 
		&& l=1,\dots,m. \label{eq:cond2}
\end{alignat} 
Note that \cref{eq:interpol_cond} and \cref{eq:cond2} can be reformulated as a linear system for the coefficient vectors $\boldsymbol{\alpha} = [\alpha_1,\dots,\alpha_K]^T$ and $\boldsymbol{\beta} = [\beta_1,\dots,\beta_m]^T$: 
\begin{equation}\label{eq:system} 
	\begin{bmatrix} \Phi & \mathrm{P} \\  \mathrm{P} ^T & 0 \end{bmatrix}
	\begin{bmatrix} \boldsymbol{\alpha} \\ \boldsymbol{\beta} \end{bmatrix} 
	= 
	\begin{bmatrix} \mathbf{u} \\ \mathbf{0} \end{bmatrix}, 
\end{equation} 
where $\mathbf{u} = [u(x_1),\dots,u(x_K)]^T$ and 
\begin{equation}\label{eq:Phi_P}
	\Phi = 
	\begin{bmatrix} 
		\varphi( | x_1 - x_1 | ) & \dots & \varphi( | x_1 - x_K | ) \\ 
		\vdots & & \vdots \\ 
		\varphi( | x_K - x_1 | ) & \dots & \varphi( | x_K - x_K | )
	\end{bmatrix}, 
	\ 
	 \mathrm{P}  = 
	\begin{bmatrix} 
		p_1(x_1) & \dots & p_m(x_1) \\ 
		\vdots & & \vdots \\  
		p_1(x_K) & \dots & p_m(x_K) 
	\end{bmatrix}.
\end{equation} 
Incorporating polynomial terms of degree up to $m-1$ in the RBF interpolant \cref{eq:RBF-interpol} is important for several reasons: 
\begin{enumerate}
    \item[(i)] The RBF interpolant \cref{eq:RBF-interpol} becomes exact for polynomials of degree up to $m-1$, i.\,e., $u^h = u$ for $u \in \mathbb{P}_{m-1}$.
    \item[(ii)] For some (conditionally positive) kernels $\varphi$, the RBF interpolant \cref{eq:RBF-interpol} only exists uniquely when polynomials up to a certain degree are incorporated.
\end{enumerate}
In addition, we will show that (i) is needed for the RBF method to be conservative \cite{glaubitz2021stabilizing,glaubitz2022nonpolnyomial}. 
The property (ii) is explained in more detail in \cref{app_polynomial} as well as in \cite[Chapter 7]{fasshauer2007meshfree} and \cite[Chapter 3.1]{glaubitz2020shock}.
For simplicity and clarity, we will focus on the choices of RBFs listed in \cref{tab:RBFs}. 
More types of RBFs and their properties can be found in the monographs \cite{wendland2004scattered,fasshauer2007meshfree,fornberg2015primer}.
\begin{table}[t]
  \centering 
  \renewcommand{\arraystretch}{1.3}
  \begin{tabular}{c|c|c|c}
    RBF & $\varphi(r)$ & parameter & order \\ \hline 
    Gaussian & $\exp( -(\varepsilon r)^2)$ & $\varepsilon>0$ & 0 \\ 
    Multiquadrics & $\sqrt{1 + (\varepsilon r)^2}$ & $\varepsilon>0$ & $1$ \\ 
    Polyharmonic splines (odd) & $r^{2k-1}$ & $k \in \N$ & $k$ \\ 
   Polyharmonic splines (even) & $r^{2k} \log r$ & $k \in \N$ & $k+1$ 
  \end{tabular} 
  \caption{Some frequently used RBFs}
  \label{tab:RBFs}
\end{table}
Note that the set of all RBF interpolants \cref{eq:RBF-interpol} forms a $K$-dimensional linear space, denoted by $\mathcal{R}_{m}(X_K)$. 
This space is spanned by the \emph{cardinal functions}  
\begin{equation}\label{eq:cardinal}
  	c_i(x) 
    = \sum_{k=1}^K \alpha_k^{(i)} \varphi( |x - x_k| ) + \sum_{l=1}^m \beta^{(i)}_l p_l(x), 
    \quad i=1,\dots,K,
\end{equation}
which are uniquely determined by the \emph{cardinal property} 
\begin{equation}\label{eq:cond_cardinal}
  c_i(x_k) = \delta_{ik} := 
  \begin{cases} 
    1 & \text{if } i=k, \\ 
    0 & \text{otherwise}, 
  \end{cases} 
  \quad i,k=1,\dots,K,
\end{equation}
and condition \cref{eq:cond2}. 
They also provide us with the following (nodal) representation of the RBF interpolant:
\begin{equation}\label{eq:basis_RBF}
	u^h(x) = \sum_{k=1}^K u(x_k) c_k(x).
\end{equation} 

\subsubsection{Radial Basis Function Methods}
\label{sub:RBF-discretization}

We outline the standard global RBF method for the IBVP \cref{eq_IBVP}. 
The domain $\Omega$ on which we solve \eqref{eq_IBVP} is discretized using two point sets:
\begin{itemize}
    \item The nodal point set (centers) $X_K=\{x_1, \cdots, x_K \}$  used for constructing the cardinal basis functions \cref{eq:cardinal}.
    \item The grid (evaluation) point set $Y_N=\{y_1, \cdots, y_N \}$ for describing the IBVP \eqref{eq_IBVP}, where $N\geq K$. 
\end{itemize}
By selecting $Y_N=X_K$, we get a collocation method, and with $N>K$, a method using oversampling.
The numerical solution  $\mathbf{u}$ is defined by the values of $u^h$ at  $Y_N$ and the operator $L(\mathbf{u})$ by using the spatial derivative of the RBF interpolant $u^h$, also at $Y_N$.
The RBF discretization can be summarized in the following three steps:
\begin{enumerate}
    \item Determine the RBF interpolant $u^h\in \mathcal{R}_{m}(X_K)$. 

    \item Define $L(\mathbf{u})$ in the semidiscrete equation by inserting \cref{eq:basis_RBF} into the continuous spatial operator. This yields
    \begin{equation}\label{eq:semidis_second_no}
    \begin{aligned}
        \operatorname{L}(\mathbf{u})=&\left(  \mathcal{L}(y_n,t, \partial_x)  u^h(t,y_n)+ \mathcal{F}(t,y_n)  \right)_{n=1}^N.\\
    \end{aligned}
    \end{equation} 
    \item Use a classical time integration scheme to evolve  \cref{eq:semidis-eq}.

\end{enumerate}

Global RBF methods come with several free parameters. 
These include the center and evaluation points $X_K$ and $Y_N$, the kernel $\varphi$, the degree $m-1$ of the polynomial term included in the RBF interpolant \cref{eq:RBF-interpol}. 
The kernel $\varphi$ might come with additional free parameters such as the shape parameter $\varepsilon$. 
Finally, we note that also the basis of the RBF approximation space $\mathcal{R}_{m}(X_K)$, that one uses for numerically computing the RBF approximation $u^h$ and its derivatives, can influence how well-conditioned the RBF method is in practice. 
Discussions of appropriate choices for these parameters are filling whole books \cite{wendland2004scattered,fasshauer2007meshfree,fornberg2011stable,fornberg2015primer} and are avoided here. 
 In this work, we have a different point in mind and focus on the  basic stability conditions of RBF methods.

%% file: 3_FSBP.tex
\section{Summation-by-parts Operators on General Function Spaces} 
\label{sec:FSBP} 

SBP operators were developed to mimic the behavior of integration by parts in the continuous setting and provide a systematic way to build energy-stable semi-discrete approximations. 
First, constructed for an underlying polynomial approximation in space, the theory was recently extended to general function spaces in \cite{glaubitz2022nonpolnyomial}. 
For completeness, we shortly review the extended framework of FSBP operators and repeat their basic properties. 
We consider the FSBP concept on the interval $[x_L, x_R]$ where the boundary points are included in the evaluation points $Y_N$. 
Using this framework, we give the following definition originally found  in \cite{glaubitz2022nonpolnyomial}:
 \begin{definition}[FSBP operators]\label{def:SBP_general} 
	Let $\mathcal{F} \subset C^1([x_L,x_R])$ be a finite-dimensional function space. 
	An operator $D = P^{-1} Q$ is an \emph{$\mathcal{F}$-based SBP (FSBP) operator} if 
	\begin{enumerate}
		\item[(i)] \label{item:SBP_general1} 
		$D f(\mathbf{x}) =f'(\mathbf{x})$ for all $f \in \mathcal{F}$, 		
		\item[(ii)] \label{item:SBP_general2} 
		$P$ is a symmetric positive definite matrix, and 
		\item[(iii)] \label{item:SBP_general3} 
		$Q + Q^T = B = \diag(-1,0,\dots,0,1)$.		
	\end{enumerate}  
\end{definition} 
Here, $f(\mathbf{y}) = [f(y_1), \dots, f(y_N)]^T$ and $f'(\mathbf{y}) = [f'(y_1),\dots, f'(y_N)]^T$ respectively denote the vector of the function values of $f$ and its derivative $f'$ at the evaluation points $y_1,\dots,y_N$. \\
Further, $D$ denotes the differentiation matrix and $P$ is a matrix defining a discrete norm. 
In order to produce an energy estimate, $P$ must be positive definite and symmetric. 
In this manuscript and in  \cite{glaubitz2022nonpolnyomial}, we focus for  stability reasons on diagonal norm FSBP operators \cite{linders2020properties,gassner2016split,ranocha2017extended}. 
The matrix $Q$ is nearly skew-symmetric and can be seen as the stiffness matrix in context of FE.
With these operators, integration-by-parts is mimicked  discretely as: 
 \begin{equation} \label{int_by_parts}
 \begin{aligned}
&& f(\mathbf{x})^T PD g(\mathbf{x})+  \left( D f(\mathbf{x}) \right)^T P g(\mathbf{x})&=
  f(\mathbf{x})^T B g(\mathbf{x}) \\
\Longleftrightarrow && \int_{x_L}^{x_R} f(x)\cdot g'(x) d x + \int_{x_L}^{x_R} f'(x)\cdot g(x) d x &=[f(x)g(x)]_{x=x_L}^{x=x_R} 
\end{aligned}
\end{equation} 
for all $f,g \in  \mathcal{F}$.

\subsection{Properties of FSBP Operators}
In \cite{glaubitz2022nonpolnyomial}, the authors proved that 
 the FSBP-SAT semi-discretization of the linear advection equation 
  yields an energy stable semi-discretization. 
The so-called SAT term  imposes the boundary condition weakly. 
Moreover, the underyling function space $\mathcal{F}$ should contain constants in order to ensure conservations. 

In context of RBF methods, constants  have to be included in the RBF interpolants \cref{eq:RBF-interpol}, also for the reasons  discussed above. \\ 
We will  extend  the previous investigation to the linear advection-diffusion equation. 
\begin{equation}\label{eq:adv_diff}
\begin{aligned}
	\partial_t u + a \partial_x u &= \partial_x (\kappa \partial_x u), \quad x \in (x_L,x_R), \ t>0,\\
       u(x,0)&= f(x),\\
       au(x_L,t)-\kappa \partial_x u(x_L,t)&=g_{x_L}(t),\\
             \kappa \partial_x u(x_R,t)&=g_{x_R}(t),
\end{aligned}
\end{equation} 
where  $a>ß$ is a constant  and $\kappa>0$ can depend on $x$ and $t$. 
The problem \eqref{eq:adv_diff} is strongly well-posed,  as can be seen by
the energy rate 
\begin{equation}\label{est_diff}
\begin{aligned}
\norm{u}_t^2+ 2 \norm{u_x}^2_{\kappa}=&a^{-1} \bigg(  g_{x_L}^2-\left(au(x_L,t)-g_{x_L}\right)^2 
 - \left(au(x_R,t)-g_{x_R}^2 \right)^2+g_{x_R}^2 \bigg)
\end{aligned}
\end{equation}
with $ \norm{u_x}^2_{\kappa}= \int_{x_L}^{x_R} (\partial_x u)^2 \kappa \mathrm{d} x.$
To translate this estimate to the discrete setting, we discretize \eqref{eq:adv_diff}. The most straightforward FSBP-SAT discretization reads 
	\begin{eq}\label{eq:lin_discr_diffusion}
	\mathbf{u}_t  +a D \mathbf{u} =  D(  \mathcal{K} D \mathbf{u} )+   P^{-1} \mathbb{S}
\end{eq} 
with  $ \mathcal{K}= \diag (\kappa)$ and 
\begin{eq} \label{eq_SAT_diff}
	\mathbb{S} := [\mathbb{S}_0,0,\dots,\mathbb{S}_1]^T, \\
	\mathbb{S}_0 := - \sigma_0 a (u_0-  (  \mathcal{K} D \mathbf{u})_0-g_{x_L}),\\
	\mathbb{S}_1 := - \sigma_1  (  (  \mathcal{K}D \mathbf{u})_N-g_{x_R}).
\end{eq}
We can prove the following result using the FSBP \cref{def:SBP_general}.
\begin{theorem}
The scheme \eqref{eq:lin_discr_diffusion} is strongly stable with $\sigma_0=-1$ and $\sigma_1=1$. 
\end{theorem}
\begin{proof}
We use the energy method together with the FSBP property to obtain 
\begin{equation}\label{eq_stable_diff}
\norm{\mathbf{u}}_t^2 +2 \norm{D\mathbf{u}}^2_{ \mathcal{K}}
= a^{-1} \left( g_{x_L}^2-(au_0-g_{x_L})^2 -(au_N-g_{x_R})^2+g_{x_R}^2 \right)
\end{equation}
with $\norm{D\mathbf{u}}^2_{ \mathcal{K}}= (D\mathbf{u})^T P \mathcal{K}D\mathbf{u}$.
This is similar to the continuous estimate \eqref{est_diff}.  Note that $P$ and $ \mathcal{K}$ have to be diagonal to ensure that $ P \mathcal{K}$ defines a norm. 
\end{proof}
Clearly, the FSBP operators automatically reproduce the results from the continuous setting, similar to the classical SBP operators  based on polynomial approximations \cite{svard2014review}. 
Note that no details are assumed on the specific function space, grid or the underlying methods. The only factors of importance is that the FSBP property is fulfilled and that well posed boundary condition are used.
In what follows, we will adapt the FSBP theory  to 
radial basis functions.

%% file: 4_RBF-SBP.tex
\section{SBP operators for RBFs} 
\label{sec:SBP_RBF} 
 
First, we adapt the FSBP theory in \cref{subsec_RBF} to the RBF framework. Next, we investigate classical RBF methods concerning the FSBP property, and demonstrate that  standard global RBF schemes does not fulfill this property. 
Finally, we describe how RBFSBP operators can be constructed
that lead to stability. 
\subsection{RBF-based SBP operators} 
The function space $\mathcal{F} \subset C^1$ for RBF methods is defined by the description in Subsection \ref{subsec_RBF}.
Consider a set of $K$ points, $X_K =\{x_1, \cdots, x_K \} \subset [x_L, x_R]$. 
The set of all RBF interpolants \cref{eq:RBF-interpol} forms a $K$-dimensional approximation space, which we denote by $\mathcal{R}_{m}(X_K)$. 
Let $\{c_k\}_{k=1}^K$ be a  basis in $\mathcal{R}_{m}(X_K)$. 
Further, we have the grid points $Y_N =\{y_1, \cdots, y_N \} \subset [x_L, x_R]$ which include the boundaries. 
They are used to define the RBFSBP operators.

\begin{definition}[RBF Summation-by-Parts Operators]\label{def:SBP_RBF} 
	An operator $D = P^{-1} Q \in \R^{N \times N}$ is an \emph{RBFSBP operator} on the grid points $Y_N$ if 
	\begin{enumerate}
		\item[(i)] \label{item:SBP_RBF1} 
		$D c_k(\mathbf{x}) = c_k'(\mathbf{x})$ for $k=1,2,\dots,K$ and $c_k \in \mathcal{R}_{m}(X_K)$, 
		\item[(ii)] \label{item:SBP_RBF2} 
		$P \in \R^{N\times N}$ is a symmetric positive definite matrix, and 		
		\item[(iii)] \label{item:SBP_RBF3} 
		$Q + Q^T = B$.		
	\end{enumerate} 
\end{definition} 
In the classical RBF discretizations, the exactness of the derivatives of the cardinal functions is  the only condition which is imposed. 
However, to construct energy stable RBF methods, \textit{the existence of an adequate norm is as important as the condition on the derivative matrix}. 
Hence it is often necessary to use a higher number of grid points than centers to ensure the existence of a positive quadrature formula to guarantee the conditions in \cref{def:SBP_RBF}.  \\
 The norm matrix $P$  in \cref{def:SBP_RBF} has only been assumed to be symmetric positive definite. 
However, as mentioned above for the remainder of this work, we restrict ourselves to diagonal norm matrices $P=\diag(\omega_1, \cdots,\omega_N)$ where  $\omega_i$ is the associated quadrature weight because Diagonal-norm operators are
\begin{enumerate}
\item[i)] required for certain splitting techniques \cite{gassner2016split, nordstrom2006conservative, offner2019error}, and variable coefficients, see for example \eqref{eq_stable_diff}.
\item[ii)]  better suited to conserve nonquadratic quantities for nonlinear stability \cite{kitson2003skew}, 
\item[iii)] easier to extend to, for instance, curvilinear coordinates \cite{chan2019efficient, ranocha2017extended, svard2004coordinate}.
\end{enumerate} 

\begin{remark}
In \cref{def:SBP_RBF}, we have two  sets of points, the interpolation points $X_K$ and the grid points $Y_N$.  The derivative matrix is constructed with respect to the exactness of the cardinal functions $c_k$ related to the interpolation points $X_K$. However, all operators  are constructed with  respect to the grid points $Y_N$, i.e.  $D,P,Q \in \R^{N\times N}$. This is in particular essential when ensuring the existence of suitable norm matrix $P$.   This means that the size of the SBP operator 
is determine by the quadrature formula. So, the number of grid points and their placing highly effects the size of the operators and so the efficiency of the underlying method itself. In the future, this will be investigated in more detail.
\end{remark}

\subsection{Existing Collocation RBF Methods and the FSBP Property} 
\label{sec:existing} 
In the classical collocation RBF approach, the centers intersect with the grid points, i.e. $X_K=Y_N$.
 It was shown in \cite{glaubitz2022nonpolnyomial} that a diagonal-norm $\mathcal{F}$-exact SBP operator exists on the grid $Y_N = \{y_1, \cdots, y_N\} $ if and only if a positive and $(\mathcal{F}\mathcal{F})'$-exact quadrature formula exists on the same grid (the same requirement as for classical SBP operators). 
The differentiation matrix $D \in \R^{N \times N}$ of a collocation RBF method can thus only satisfy the FSBP property if there exists a positive and $(\mathcal{R}_{m}(Y_N) \mathcal{R}_{m}(Y_N))'$-exact quadrature formula on the grid $Y_N$. 
The weights $\mathbf{w} \in \R^N$ of such a quadrature formula would have to satisfy 
\begin{equation}\label{eq:cond_quadrature}
	G \mathbf{w} = \mathbf{m}, \quad \mathbf{w} > 0,
\end{equation}
with the coefficient matrix $G$ and vector of moments $\mathbf{m}$ given by 
\begin{equation} \label{eq_cond_quadrature_2}
	G = 
	\begin{bmatrix} 
		g_1(y_1) & \dots & g_1(y_N) \\ 
		\vdots & & \vdots \\ 
		g_L(y_1) & \dots & g_L(y_N) 
	\end{bmatrix}, \quad 
	\mathbf{m} = 
	\begin{bmatrix} 
		\int_a^b g_1(y) \intd y\\ 
		\vdots \\ 
		\int_a^b g_L(y) \intd y
	\end{bmatrix},
\end{equation} 
In \eqref{eq_cond_quadrature_2},  $\{g_l\}_{l=1}^L$ is a basis of the function space $(\mathcal{R}_{m}(Y_N) \mathcal{R}_{m}(Y_N))'$. 
In many cases, the dimension $L$ of $(\mathcal{R}_{m}(Y_N) \mathcal{R}_{m}(Y_N))'$ is larger than the dimension $N$  of $\mathcal{R}_{m}(Y_N)$. 
In this case, $L > N$ and the linear system in \cref{eq:cond_quadrature} is overdetermined and has no solution.  
This is demonstrated in Table \ref{tab:SBP_existing_diag}, which reports on the residual and smallest element of the least squares solution (solution with minimal $\ell^2$-error) of  \cref{eq:cond_quadrature} for different cases. 
In all of our considered tests, the  residuals were always larger than zero indicating that the operator is not in SBP form. Similar results are obtained for non-diagonal norm matrices $P$, which is outlined in \cref{sec:existing_nondiagonal}. 
\begin{table}[tb]
\renewcommand{\arraystretch}{0.8}
\centering 
  	\begin{tabular}{c c c c c c c c c} 
		\toprule 
		\multicolumn{9}{c}{Equidistant points} \\ \hline 
		& & \multicolumn{3}{c}{$\| G\mathbf{w}-\mathbf{m} \|_2$} & & \multicolumn{3}{c}{$\min \mathbf{w}$} \\ \hline
		$N$/$m-1$ & & $0$ & $1$ & $2$ & & $0$ & $1$ & $2$ \\ \hline 
		$10$ & 
			& $7.6\cdot 10^{-1}$ & $6.6\cdot 10^{-1}$ & $1.3\cdot 10^{-12}$ &
			& $2.7\cdot 10^{-2}$ & $3.3\cdot 10^{-2}$ & $5.6\cdot 10^{-2}$ \\ 
		$20$ & 
			& $6.9\cdot 10^{-1}$ & $6.4\cdot 10^{-1}$ & $6.1\cdot 10^{-11}$ &
			& $1.5\cdot 10^{-2}$ & $1.6\cdot 10^{-2}$ & $2.6\cdot 10^{-2}$ \\ 
		$40$ & 
			& $6.6\cdot 10^{-1}$ & $6.4\cdot 10^{-1}$ & $2.5\cdot 10^{-9}$ &
			& $7.7\cdot 10^{-3}$ & $8.0\cdot 10^{-3}$ & $1.3\cdot 10^{-2}$ \\  \hline \\
	\end{tabular} 
	\\
	\begin{tabular}{c c c c c c c c c} 
		\multicolumn{9}{c}{Halton points} \\ \hline 
		& & \multicolumn{3}{c}{$\| G\mathbf{w}-\mathbf{m} \|_2$} & & \multicolumn{3}{c}{$\min \mathbf{w}$} \\ \hline
		$N$/$m-1$ & & $0$ & $1$ & $2$ & & $0$ & $1$ & $2$ \\ \hline 
		$10$ & 
			& $1.0$ & $1.0$ & $5.6$ &
			& $3.1\cdot 10^{-4}$ & $3.9\cdot 10^{-4}$ & $5.6\cdot 10^{-3}$ \\ 
		$20$ & 
			& $1.0$ & $1.0$ & $1.0\cdot 10^{1}$ &
			& $2.5\cdot 10^{-6}$ & $2.7\cdot 10^{-6}$ & $-4.3\cdot 10^{-3}$ \\ 
		$40$ & 
			& $1.0$ & $1.0$ & $1.6\cdot 10^{1}$ &
			& $2.2\cdot 10^{-10}$ & $2.3\cdot 10^{-10}$ & $-1.3\cdot 10^{-3}$ \\  \hline \\
	\end{tabular} 
	\\
	\begin{tabular}{c c c c c c c c c} 
		\multicolumn{9}{c}{Random points} \\ \hline 
		& & \multicolumn{3}{c}{$\| G\mathbf{w}-\mathbf{m} \|_2$} & & \multicolumn{3}{c}{$\min \mathbf{w}$} \\ \hline
		$N$/$m-1$ & & $0$ & $1$ & $2$ & & $0$ & $1$ & $2$ \\ \hline 
		$10$ & 
			& $1.3$ & $1.2$ & $1.5\cdot 10^{1}$ &
			& $1.1\cdot 10^{-6}$ & $1.3\cdot 10^{-6}$ & $-9.6\cdot 10^{-2}$ \\ 
		$20$ & 
			& $1.1$ & $1.1$ & $1.1\cdot 10^{2}$ &
			& $5.6\cdot 10^{-16}$ & $1.8\cdot 10^{-15}$ & $-1.7\cdot 10^{-1}$ \\ 
		$40$ & 
			& $1.3$ & $1.3$ & $1.8\cdot 10^{3}$ &
			& $-4.1\cdot 10^{-11}$ & $-2.9\cdot 10^{-11}$ & $-1.2\cdot 10^{1}$ \\ \bottomrule
	\end{tabular} 
\caption{Residual $\| G\mathbf{w}-\mathbf{m} \|_2$ and smallest elements $\min \mathbf{w}$ for the cubic PHS-RBF on equidistant, Halton, and random points. }\label{tab:SBP_existing_diag}
\end{table}


\subsection{Existence and Construction of RBFSBP Operators} 

Translating the main result from  \cite{glaubitz2022nonpolnyomial},  
 we need quadrature formulas to ensure the exact integration of  $(\mathcal{R}_{m}(X_K) \mathcal{R}_{m}(X_K))'$. 
For RBF spaces, we use least-squares formulas, which can be used on almost arbitrary sets of grid points $Y_N$ and to any degree of exactness. 
The least squares ansatz always leads to a positive and $(\mathcal{R}_{m}(X_K) \mathcal{R}_{m}(X_K))'$-exact quadrature formula as long a sufficiently large number of data points $Y_N$ is used. 
\begin{remark} 
	Existing results on positivity and exactness of least squares quadrature formulas usually assume that the function space contains constants \cite{glaubitz2020stableQF,glaubitz2021constructing}. 
	Translating this to our setting, we need this property to be fulfilled for 
	$(\mathcal{R}_{m}(X_K) \mathcal{R}_{m}(X_K))'$. Therefore, 
	 $\mathcal{R}_{m}(X_K)$ should contain constants and linear functions. 
	However, this assumption is primarily made for technical reasons and can be relaxed. 
	Indeed, even when $\mathcal{R}_{m}(X_K)$ only contained constants, we were still able to construct positive and $(\mathcal{R}_{m}(X_K) \mathcal{R}_{m}(X_K))'$-exact least squares quadrature formulas in all our examples. 
	Future work will provide a theoretical justification for this.  
\end{remark} 
Due to the least-square ansatz, we may always assume that 
we have a positive and \\$(\mathcal{R}_{m}(X_K) \mathcal{R}_{m}(X_K))'$-exact quadrature formula. With that ensured,  we summarize the algorithm to construct a diagonal norm RBFSBP operators in the following steps:
\begin{enumerate}
\item Build $P$ by setting the quadrature weights on the diagonal.
\item Split $Q$ into its known symmetric $\frac{1}{2}B$ and unknown anti-symmetric part $Q_A$.
\item Calculate $Q_A$ by using 
$$
Q_AC =P C_x -\frac{1}{2} BC \text{ with }C = [c_1(\mathbf{y}), \dots, c_K(\mathbf{y})] =
		\begin{bmatrix} 
			c_1(y_1) & \dots &c_K(y_1) \\ 
			\vdots & & \vdots \\ 
			c_1(y_N) & \dots & c_K(y_N)
		\end{bmatrix}$$ 
and $C_x = [c_1'(\mathbf{y}), \dots, c_K'(\mathbf{y})]$ is defined analogous to $C$ where
$\{c_1,...,c_K\}$ is  a basis of the K-dimensional function space.
\item Use $Q_A$ in  $Q= Q_A+\frac{1}{2} B$ to calculate $Q$. 
\item $D=P^{-1}Q$ gives the RBFSBP operator. 
\end{enumerate}
In the RBF context,  one can always use cardinal functions as the basis. However, for simplicity reason is can be wise to use another basis representation, derived from the cardinal functions.

%% file: 5_SBP_RBF_Operators.tex
\section{RBFSBP Operators}\label{sec_Operators}
Next, we construct  RBFSBP operators for a few frequently used kernels\footnote{The matlab code to replicate the results is provided in the corresponding repository \url{https://github.com/phioeffn/Energy_stable_RBF}.}.
We consider a set of $K$ points, ${X_K = \{x_1,\dots,x_K\} \subset [x_L,x_R]}$, and assume that these include the boundaries $x_L$ and $x_R$. 
Henceforth, we will consider the kernels listed in Table \ref{tab:RBFs} and augment them with constants.
The set of all RBF interpolants including constants \cref{eq:RBF-interpol} forms a $K$-dimensional approximation space, which we denote by $\mathcal{R}_{1}(X_K)$. 
This space is spanned by the \emph{cardinal functions} $c_k \in \mathcal{R}_{1}(X_K)$ which are uniquely determined by \eqref{eq:cond_cardinal}. 
The matching constraint is then simply	
$
\sum_{k=1}^K \alpha_k  = 0. 
$
That is, 
\begin{equation}\label{eq:cubicRBF_space}
	\mathcal{R}_{1}(X_K) = \Span\{ \, c_k \mid k=1,\dots,K \, \}
\end{equation}
with the approximation space $\mathcal{R}_{1}(X_K)$ having dimension $K$. \\
The product space $\mathcal{R}_{1}(X_K)\mathcal{R}_{1}(X_K)$ and its derivative space $(\mathcal{R}_{1}(X_K)\mathcal{R}_{m}(X_K))'$ are respectively given by 
\begin{align} 
	\mathcal{R}_{1}(X_K)\mathcal{R}_{1}(X_K) & = \Span\{ \, c_k c_l \mid k,l=1,\dots,K \, \}, \label{eq:FF_cubic} \\ 
	(\mathcal{R}_{1}(X_K)\mathcal{R}_{1}(X_K))' & = \Span\{ \, c_k' c_l + c_k c_l' \mid k,l=1,\dots,K \, \}. \label{eq:DFF_cubic}
\end{align} 
Note that the right-hand sides of \cref{eq:FF_cubic} and \cref{eq:DFF_cubic} both use $K^2$ elements to span the product space $\mathcal{R}_{1}(X_K)\mathcal{R}_{1}(X_K)$ and its derivative space $(\mathcal{R}_{1}(X_K)\mathcal{R}_{1}(X_K))'$. 
However, these elements are not linearly independent and the dimensions of $\mathcal{R}_{1}(X_K)\mathcal{R}_{1}(X_K)$ and $(\mathcal{R}_{1}(X_K)\mathcal{R}_{1}(X_K))'$ are smaller than $K^2$. 
Indeed, we can observe that $c_k c_l = c_l c_k$ and the dimension of \cref{eq:FF_cubic} is therefore bounded from above by 
\begin{equation} 
	\dim \mathcal{R}_{1}(X_K)\mathcal{R}_{1}(X_K) \leq \frac{K(K+1)}{2}. 
\end{equation}
Finally, we point out that in the calculation of the operators $P, Q$ and $D$ below, we will round the numbers to the second decimal place.

\subsection{RBFSBP Operators using Polyharmonic Splines}
In the first test, we work with  cubic  polyharmonic splines, $\varphi(r) = r^3$.
On $[x_L,x_R] = [0,1]$ and for the centers $X_3 = \{0,1/2,1\}$, the three-dimensional cubic RBF approximation space \cref{eq:cubicRBF_space} is given by 
$
	\mathcal{R}_{1}(X_3)
		= \Span\{ \, c_1, c_2, c_3 \, \} 
		= \Span\{ \, b_1, b_2, b_3 \, \}
$ 
with cardinal functions 
\begin{equation}
\begin{aligned}
	c_1(x) 
		& = \frac{1}{2} |x|^3 - 2 | x - 1/2 |^3 + \frac{3}{2} |x-1|^3 - \frac{1}{4}, \\ 
	c_2(x) 
		& = -2 |x|^3 + 4 | x - 1/2 |^3 - 2 |x-1|^3 + \frac{3}{2}, \\ 
	c_3(x) 
		& = \frac{3}{2} |x|^3 - 2 | x - 1/2 |^3 + \frac{1}{2} |x-1|^3 - \frac{1}{4} 
\end{aligned} 
\end{equation} 
and alternative basis functions\footnote{This basis can be constructed using a simple Gauss elimination method. }
\begin{align*}
	b_1(x) = 1, \quad 
	b_2(x) = x^3 - |x-1/2|^3, \quad
	b_3(x) = x^3 + (x-1)^3. 
\end{align*} 
We make the transformation to the basis representation $\Span \{b_1,b_2, b_3\}$ to simplify the determination of $(\mathcal{R}_{1}(X_3)\mathcal{R}_{1}(X_3))'$. In this alternative basis representation,
the product space $\mathcal{R}_{1}(X_3)\mathcal{R}_{1}(X_3)$ and its derivative space $(\mathcal{R}_{1}(X_3)\mathcal{R}_{1}(X_3))' $ are respectively given by 
\begin{equation} 
\begin{aligned} 
	\mathcal{R}_{1}(X_3)\mathcal{R}_{1}(X_3)
		& = \Span\{ \, 1, b_2, b_3, b_2^2, b_3^2, b_2 b_3 \, \} \\ 
	(\mathcal{R}_{1}(X_3)\mathcal{R}_{1}(X_3))' 
		& = \Span\{ \, b_2', b_3', b_2' b_2, b_3' b_3, b_2' b_3 + b_2 b_3' \, \}. 
\end{aligned}
\end{equation} 
Next, we have to find an  $(\mathcal{R}_{1}(X_3)\mathcal{R}_{1}(X_3))'$-exact quadrature formula with positive weights. 
For the chosen $N=4$ equidistant grid points,  the least-squares quadrature formula has positive weights  and is  $(\mathcal{R}_{1}(X_3)\mathcal{R}_{1}(X_3))'$-exact. The points and weights are
	$\mathbf{x} = \left[0, \frac{1}{3}, \frac{2}{3}, 1 \right]^T$ and $ 
	P = \diag\left( \frac{16}{129}, \frac{81}{215}, \frac{81}{215}, \frac{16}{129}\right)$.
The corresponding matrices $Q$ and $D$ of the RBFSBP operator $D = P^{-1} Q$ obtained from the construction procedure described before are 
{\small{
\begin{equation} \label{eq:operators_1}
\renewcommand*{\arraystretch}{1.2} 
	Q \approx
	 \left(\begin{array}{cccc} -\frac{1}{2} & \frac{59}{100} & -\frac{3}{20} & \frac{3}{50}\\ -\frac{59}{100} & 0 & \frac{37}{50} & -\frac{3}{20}\\ \frac{3}{20} & -\frac{37}{50} & 0 & \frac{59}{100}\\ -\frac{3}{50} & \frac{3}{20} & -\frac{59}{100} & \frac{1}{2} \end{array}\right), \quad
	D \approx
 \left(\begin{array}{cccc} -\frac{403}{100} & \frac{473}{100} & -\frac{121}{100} & \frac{51}{100}\\ -\frac{39}{25} & 0 & \frac{49}{25} & -\frac{2}{5}\\ \frac{2}{5} & -\frac{49}{25} & 0 & \frac{39}{25}\\ -\frac{51}{100} & \frac{121}{100} & -\frac{473}{100} & \frac{403}{100} \end{array}\right).
\end{equation}}}
\normalsize 
This  example was presented with less details in \cite{glaubitz2022nonpolnyomial}.

\subsection{RBFSBP Operators using Gaussian Kernels} \label{se_Gaussian}
Next, we consider the  Gaussian kernel $ \varphi(r) = \exp(-r^2)$
on  $[x_L,x_R] = [0,1]$ for the centers $X_3 = \{0,1/2,1\}$. The three-dimensional Gaussian RBF approximation space \cref{eq:cubicRBF_space} is given by 
$
	\mathcal{R}_{1}(X_3)
		= \Span\{ \, c_1, c_2, c_3 \, \} 
$
with cardinal functions 
\begin{equation}
\begin{aligned}
	c_1(x) 
		& = 2.7698 \exp(-x^2) -3.9576 \exp(-(x-0.5)^2) + 1.1878 \exp(-(x-1)^2) +0.8754\\
	c_2(x) 
		& = -3.9576 \exp(-x^2)  +7.9153 \exp(-(x-0.5)^2) -3.9576\exp(-(x-1)^2) -0.7509 \\ 
	c_3(x) 
		& =1.1878 \exp(-x^2)  -3.9576\exp(-(x-0.5)^2) + 2.7698 \exp(-(x-1)^2)+0.87543
\end{aligned} 
\end{equation} 
Again for $N=4$ equidistant grid points in  the least square quadrature formula,  we obtain exactness and positive weights. They are $
	\mathbf{x} = \left[0, \frac{1}{3}, \frac{2}{3}, 1 \right]^T$ and $
	P = \diag \left( 0.15, 0.36, 0.36,0.15  \right)$.
	The corresponding matrices $Q$ and $D$ of the RBFSBP operator $D = P^{-1} Q$ obtained from the construction procedure described before are 
{\small{\begin{equation} 
\renewcommand*{\arraystretch}{1.2} 
	Q \approx
	\left(\begin{array}{cccc} -\frac{1}{2} & \frac{3}{5} & -\frac{3}{100} & -\frac{7}{100}\\ -\frac{3}{5} & 0 & \frac{16}{25} & -\frac{3}{100}\\ \frac{3}{100} & -\frac{16}{25} & 0 & \frac{3}{5}\\ \frac{7}{100} & \frac{3}{100} & -\frac{3}{5} & \frac{1}{2} \end{array}\right),
	\quad 
	D \approx
\left(\begin{array}{cccc} -\frac{33}{10} & \frac{397}{100} & -\frac{23}{100} & -\frac{9}{20}\\ -\frac{42}{25} & 0 & \frac{89}{50} & -\frac{1}{10}\\ \frac{1}{10} & -\frac{89}{50} & 0 & \frac{42}{25}\\ \frac{9}{20} & \frac{23}{100} & -\frac{397}{100} & \frac{33}{10} \end{array}\right).
\end{equation}}
\normalsize 
To include an example with non-equidistant points for the centers, we also build matrices and FSBP operators with Halton points 
 $X_3$ for this case. A bit surprising, we need twice as many  points than on an equidistant grid to get a positive exact quadrature formula. We obtain an exact quadrature using  the nodes and weights  
$	\mathbf{x} = \left[i/7, \right]^T,$ with $i=0,\cdots,7,$ and $
	P =  \diag \left(   0.04, 0.12, 0.19, 0.13, 0. 04, 0.10, 0.30, 0.08 \right)$.
The corresponding matrices $Q$ and $D$ are $\R^{8 \times 8}$ and are given by 
{\small{
\begin{align*} 
\renewcommand*{\arraystretch}{1.1} 
	Q \approx
\left(\begin{array}{cccccccc} -\frac{1}{2} & \frac{33}{100} & \frac{29}{100} & \frac{7}{100} & -\frac{7}{100} & -\frac{2}{25} & -\frac{19}{100} & \frac{3}{20}\\ -\frac{33}{100} & 0 & \frac{11}{100} & \frac{1}{10} & \frac{7}{100} & \frac{2}{25} & \frac{3}{50} & -\frac{1}{10}\\ -\frac{29}{100} & -\frac{11}{100} & 0 & \frac{9}{100} & \frac{1}{10} & \frac{13}{100} & \frac{11}{50} & -\frac{13}{100}\\ -\frac{7}{100} & -\frac{1}{10} & -\frac{9}{100} & 0 & \frac{3}{100} & \frac{3}{50} & \frac{23}{100} & -\frac{3}{50}\\ \frac{7}{100} & -\frac{7}{100} & -\frac{1}{10} & -\frac{3}{100} & 0 & \frac{1}{100} & \frac{4}{25} & -\frac{1}{20}\\ \frac{2}{25} & -\frac{2}{25} & -\frac{13}{100} & -\frac{3}{50} & -\frac{1}{100} & 0 & \frac{1}{10} & \frac{1}{10}\\ \frac{19}{100} & -\frac{3}{50} & -\frac{11}{50} & -\frac{23}{100} & -\frac{4}{25} & -\frac{1}{10} & 0 & \frac{59}{100}\\ -\frac{3}{20} & \frac{1}{10} & \frac{13}{100} & \frac{3}{50} & \frac{1}{20} & -\frac{1}{10} & -\frac{59}{100} & \frac{1}{2} \end{array}\right),
\\
\renewcommand*{\arraystretch}{1.1} 
D \approx
\left(\begin{array}{cccccccc} -\frac{304}{25} & \frac{811}{100} & \frac{177}{25} & \frac{41}{25} & -\frac{7}{4} & -\frac{197}{100} & -\frac{451}{100} & \frac{71}{20}\\ -\frac{137}{50} & 0 & \frac{91}{100} & \frac{17}{20} & \frac{29}{50} & \frac{33}{50} & \frac{53}{100} & -\frac{79}{100}\\ -\frac{157}{100} & -\frac{59}{100} & 0 & \frac{23}{50} & \frac{14}{25} & \frac{69}{100} & \frac{29}{25} & -\frac{71}{100}\\ -\frac{27}{50} & -\frac{83}{100} & -\frac{69}{100} & 0 & \frac{21}{100} & \frac{12}{25} & \frac{46}{25} & -\frac{47}{100}\\ \frac{167}{100} & -\frac{33}{20} & -\frac{239}{100} & -\frac{31}{50} & 0 & \frac{29}{100} & \frac{191}{50} & -\frac{113}{100}\\ \frac{81}{100} & -\frac{81}{100} & -\frac{32}{25} & -\frac{3}{5} & -\frac{3}{25} & 0 & \frac{99}{100} & \frac{101}{100}\\ \frac{31}{50} & -\frac{11}{50} & -\frac{73}{100} & -\frac{77}{100} & -\frac{11}{20} & -\frac{33}{100} & 0 & \frac{99}{50}\\ -\frac{87}{50} & \frac{23}{20} & \frac{157}{100} & \frac{7}{10} & \frac{29}{50} & -\frac{6}{5} & -\frac{351}{50} & \frac{597}{100} \end{array}\right)
\end{align*}}}
\normalsize

\subsection{RBFSBP Operators using Multiquadric Kernels}
As the last example, we consider  the SBPRBF operators using  multiquadric kernels $ \varphi(r) = \sqrt{1+r^2}$ 
on  $[x_L,x_R] = [0,0.5]$ and centers $X_3 = \{0,1/4,1/2\}$. The  $(\mathcal{R}_{1}(X_3)\mathcal{R}_{1}(X_3))'$-exact  least square ansatz yields the points $\mathbf{x} = \left[0, \frac{1}{6}, \frac{1}{3}, \frac12\right]^T$ and norm matrix 
$
	P = \diag \left( 0.07, 0.18, 0.18,0.07  \right).
$
 With this norm matrix, we obtain finally
 {\small{
 \begin{equation} 
\renewcommand*{\arraystretch}{1.2} 
Q \approx
\left(\begin{array}{cccc} -\frac{1}{2} & \frac{57}{100} & -\frac{1}{50} & -\frac{1}{20}\\ -\frac{57}{100} & 0 & \frac{59}{100} & -\frac{1}{50}\\ \frac{1}{50} & -\frac{59}{100} & 0 & \frac{57}{100}\\ \frac{1}{20} & \frac{1}{50} & -\frac{57}{100} & \frac{1}{2} \end{array}\right)
\quad
	D \approx
\left(\begin{array}{cccc} -\frac{767}{100} & \frac{219}{25} & -\frac{29}{100} & -\frac{79}{100}\\ -\frac{309}{100} & 0 & \frac{319}{100} & -\frac{1}{10}\\ \frac{1}{10} & -\frac{319}{100} & 0 & \frac{309}{100}\\ \frac{79}{100} & \frac{29}{100} & -\frac{219}{25} & \frac{767}{100} \end{array}\right)
\end{equation} }}

%% file: num_test.tex
\section{Numerical Results} 
\label{sec:num_tests} 
\normalsize 
For all numerical tests presented in this work, we used an explicit SSP-RK methods. The step size $\Delta t$ was chosen to be sufficiently small. To guarantee stability, we applied weakly enforced boundary conditions using  Simultanuous Approximation Terms (SATs), as is usually done in the SBP community \cite{abgrall2020analysis,abgrall2021analysis}, and for RBFs in \cite{glaubitz2021towards}. 
To avoid matrices with high condition number, we sometimes  use a multi-block structure in our tests. In each block, a global RBF method is used and the blocks are coupled using  SAT terms as in \cite{carpenter2010revisting,  gong2011intrerface}.
We mainly use polyharmonic splines in the upcoming tests.

\subsection{Advection with Periodic Boundary Conditions}
In the first test, we consider the linear advection 
\begin{equation}\label{eq:linear}
	\partial_t u + a \partial_x u = 0, \quad x \in (x_L,x_R), \ t>0,
\end{equation} 
with $a=1$ and periodic BCs. 
The initial condition is  $u(x,0) =\mathrm{e}^{-20x^2} $ from the introducing example \eqref{eq_simple_example}  and the domain is  $[-1,1]$. 
We are in the same setting as shown in \cref{fig:init_approx}.
We compare a classical collocation RBF method with our new RBFSBP methods, focus on cubic splines and consider the final time to be $T=2$. In \cref{fig:coll_1} and \cref{fig:coll_3},  the solutions are plotted using collocation RBF method and the  RBFSBP approach. In    \cref{fig:coll_1}, we select $K=15$  for both approximations. The collocation RBF method damp the Gaussian bump significantly while the RBFSBP  method do better. The decrease can also be seen in the energy profile  \ref{fig:coll_2} where the collocation approach lose more. 
To obtain a comparable result between the collocation and RBFSBP methods, we double the number of interpolation points $K$ in our second simulation for the collocation RBF method, cf.  \cref{fig:coll_3} and  \cref{fig:coll_4}. The RBFSBP method still  performs better and demonstrates the advantage of the RBFSBP approach.\\
\begin{figure}[tb]
	\centering 
	\begin{subfigure}[b]{0.33\textwidth}
		\includegraphics[width=\textwidth]{%
      		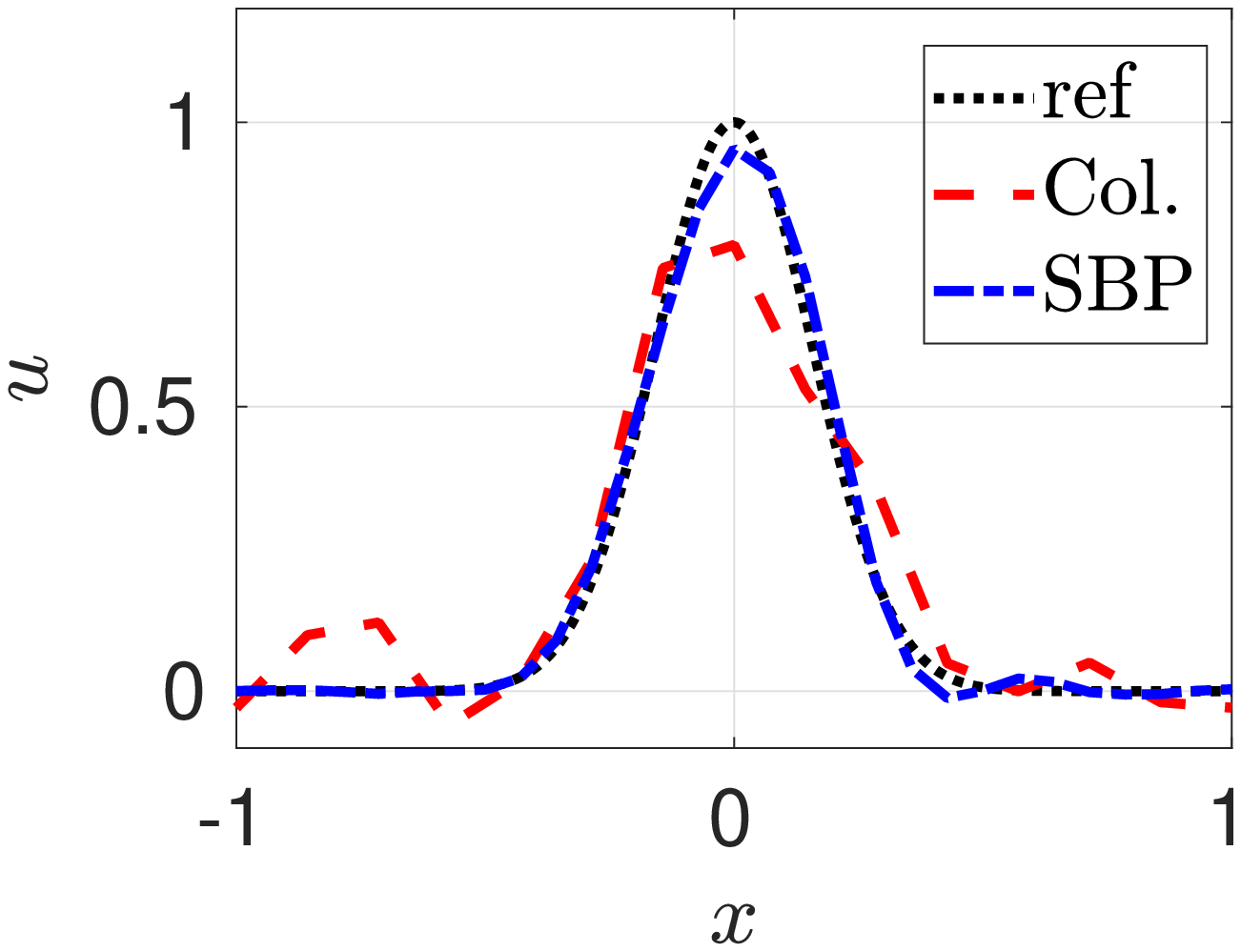} 
    		\caption{$T=2$, $K=15$}
    		\label{fig:coll_1}
  	\end{subfigure}%
	~
  	\begin{subfigure}[b]{0.33\textwidth}
		\includegraphics[width=\textwidth]{%
      		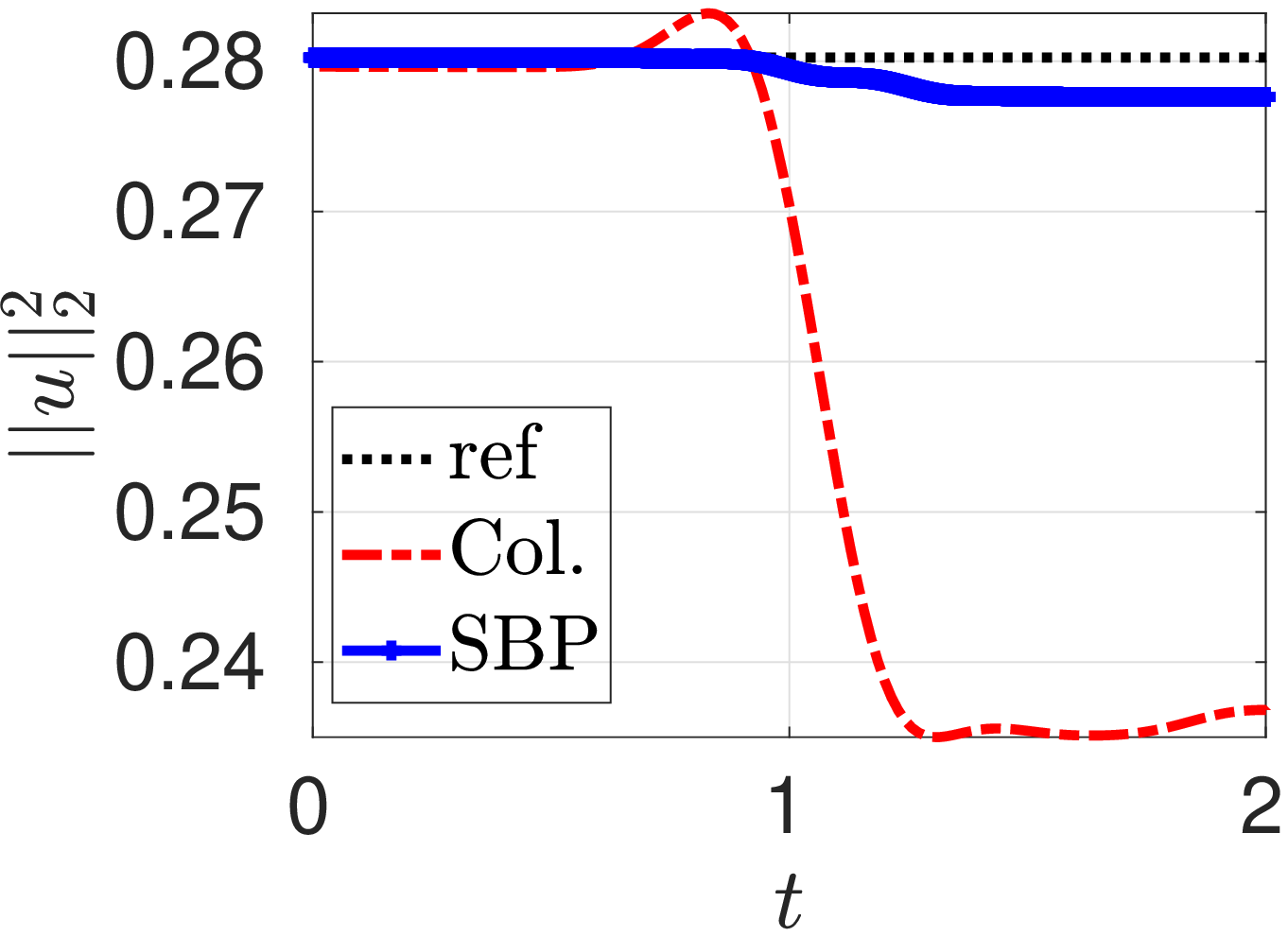} 
    		\caption{$T=2$, $K=15$}
    		\label{fig:coll_2}
  	\end{subfigure}%
	\\
		\begin{subfigure}[b]{0.33\textwidth}
		\includegraphics[width=\textwidth]{%
      		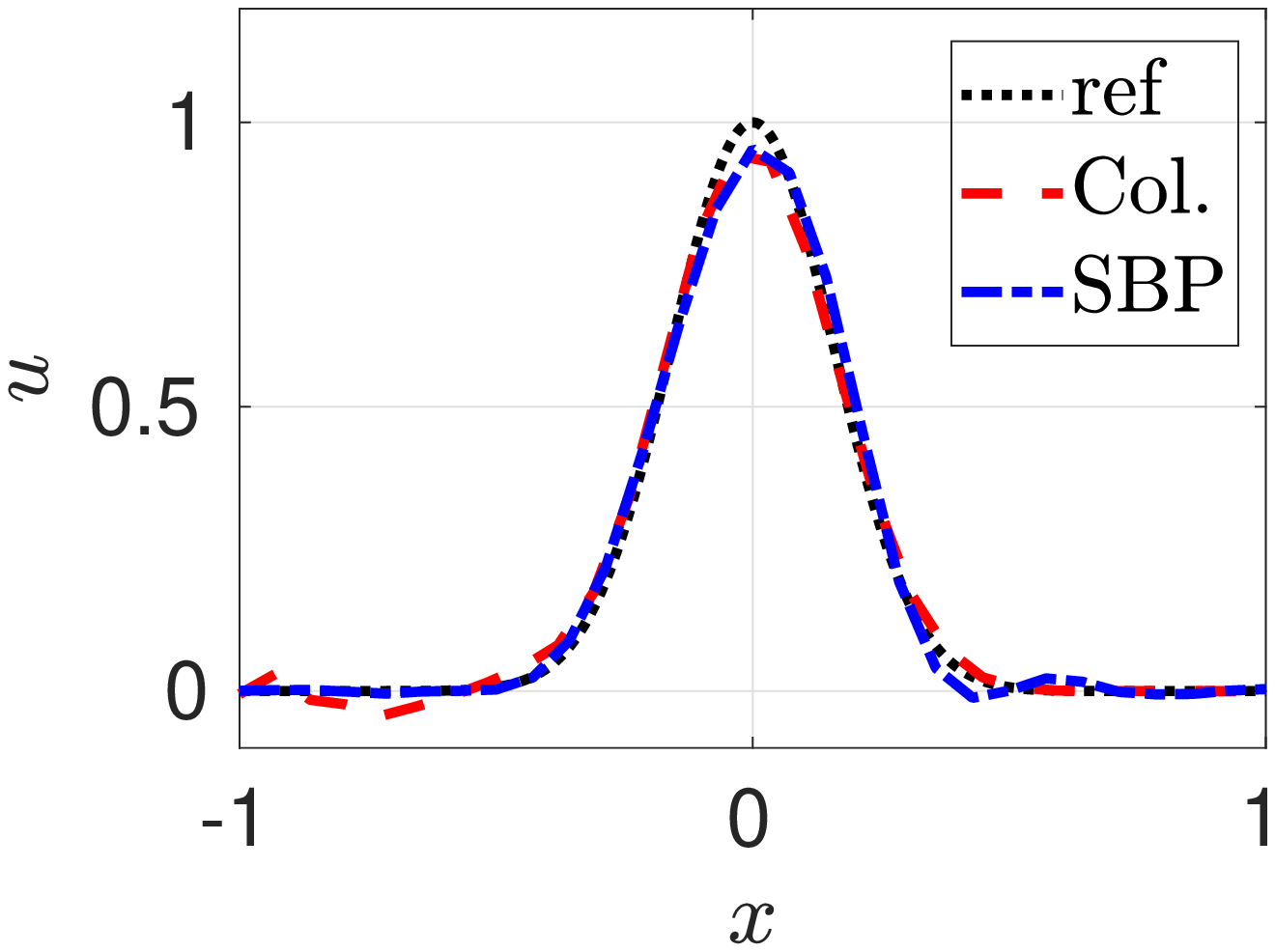} 
    		\caption{$K=15/30$}
    		\label{fig:coll_3}
  	\end{subfigure}%
	~
  	\begin{subfigure}[b]{0.33\textwidth}
		\includegraphics[width=\textwidth]{%
      		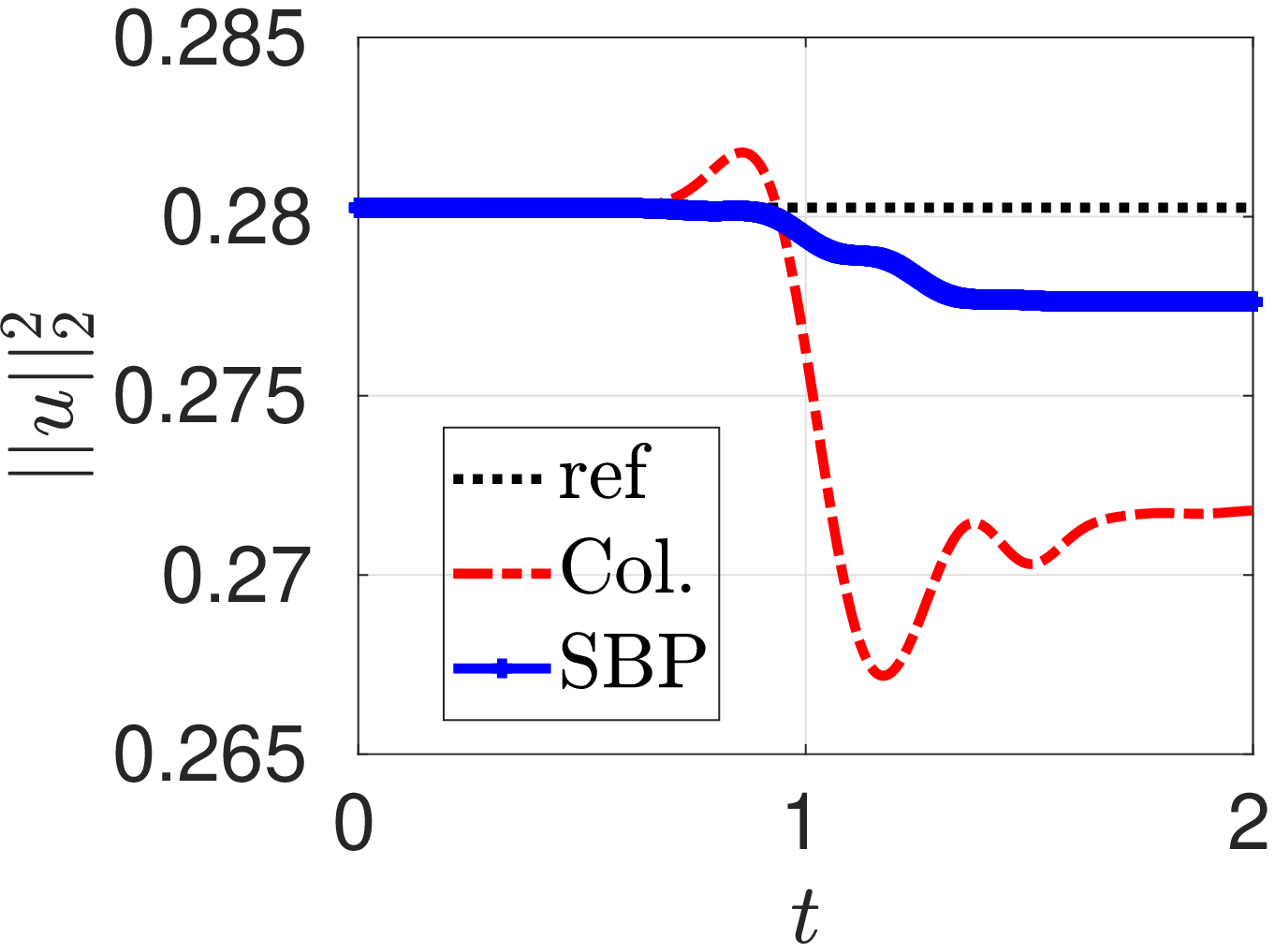} 
    		\caption{$K=15/30$}
    		\label{fig:coll_4}
  	\end{subfigure}%
  	\caption{
  	 Cubic kernels with approximation spaces $K=15/30$ on equidistant points after 1 period
	}
  	\label{fig:init_collocation}
\end{figure} 
Next, we focus only on RBFSBP methods and demonstrate the high accuracy of the approach by increasing the degrees of freedom.  In \cref{fig:init_approx2}, we plot the result and the energy using Gaussian ($\epsilon=1$) and cubic kernels. We use $K=5$ and $I=20$ blocks. 
 We obtain an  highly accurate solution and the energy remains constant. 
\begin{figure}[tb]
	\centering 
	\begin{subfigure}[b]{0.4\textwidth}
		\includegraphics[width=\textwidth]{%
      		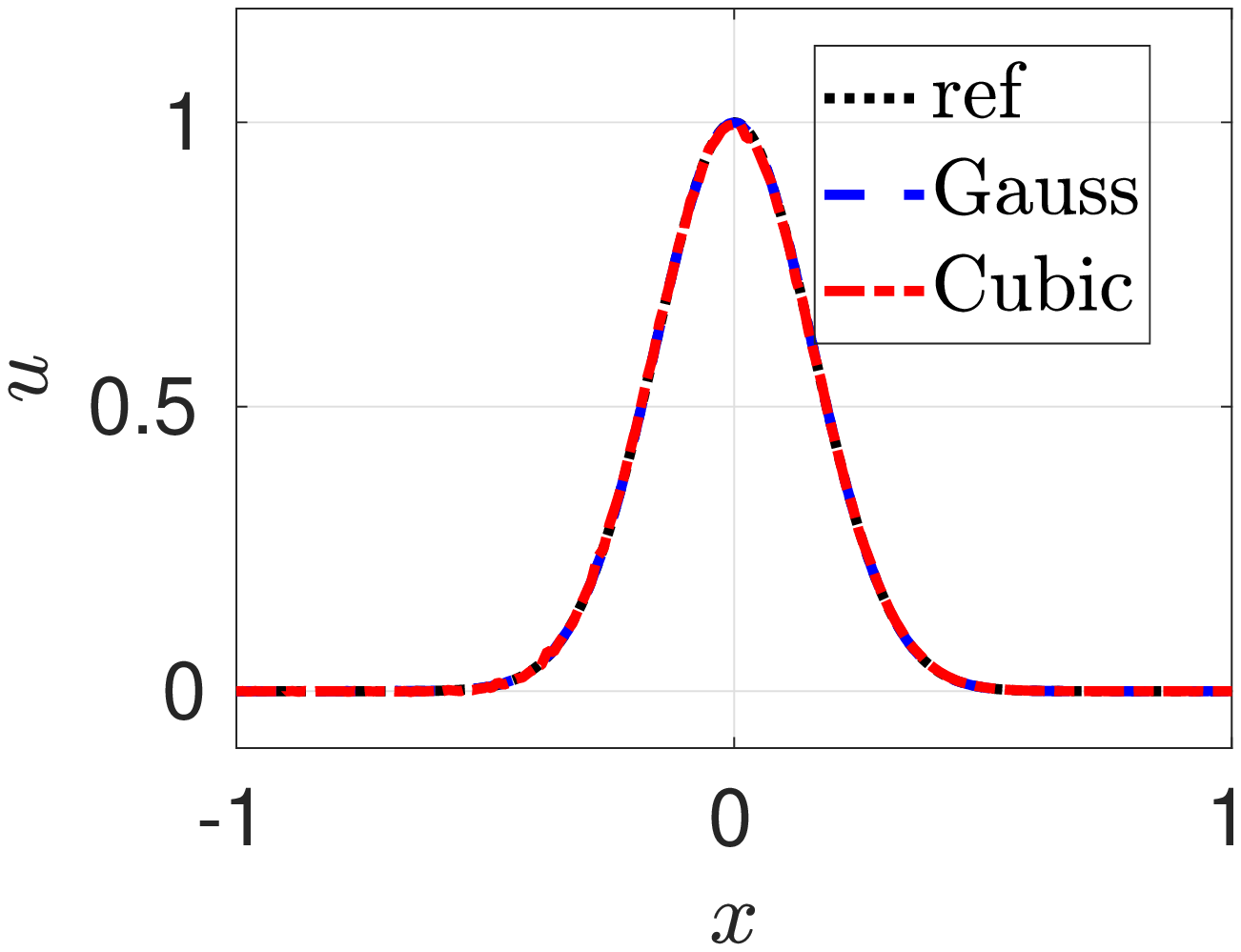} 
    		\caption{Numerical solution at $t=10$}
    		\label{fig:init_solution2}
  	\end{subfigure}%
	~
  	\begin{subfigure}[b]{0.25\textwidth}
		\includegraphics[width=\textwidth]{%
      		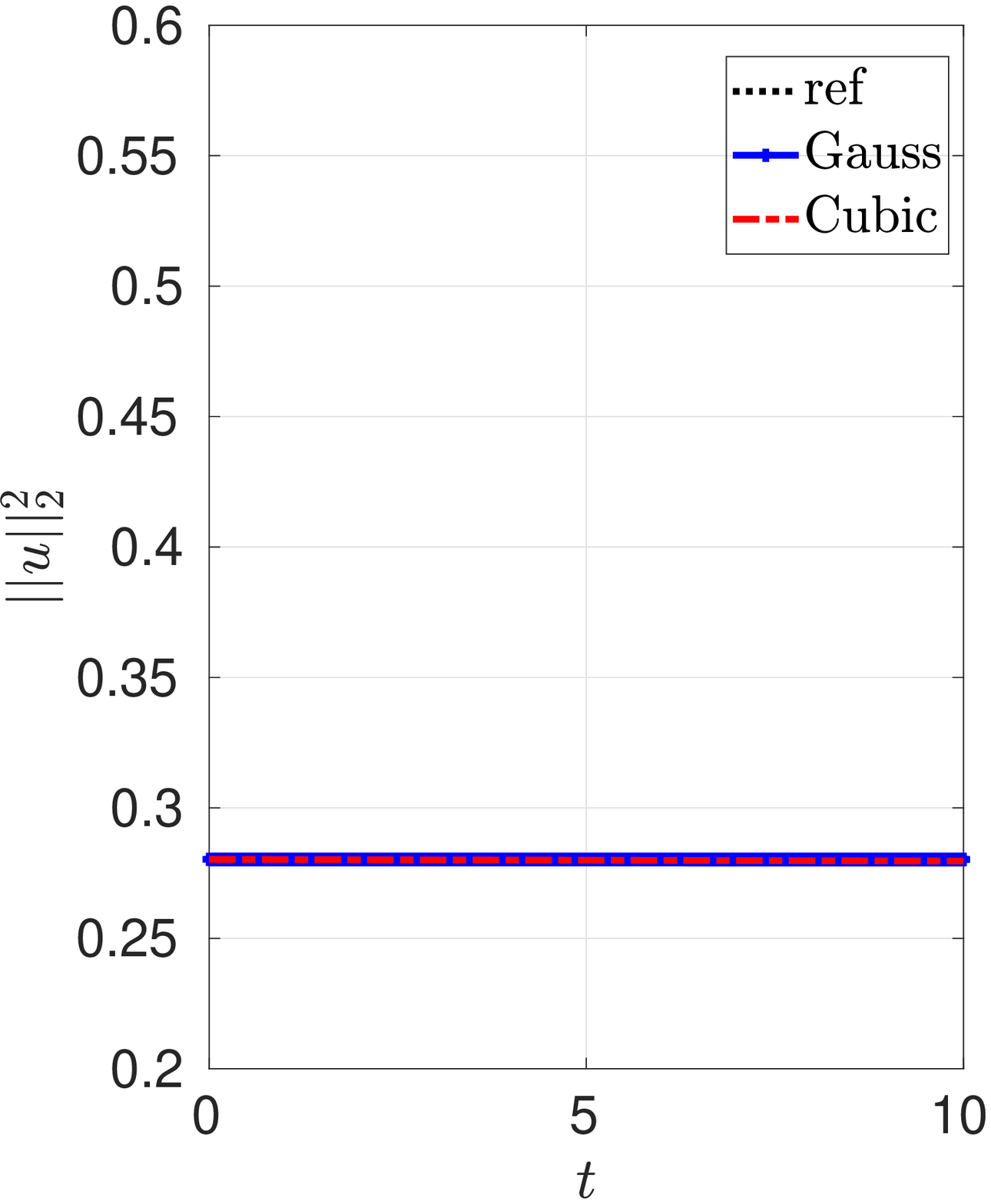} 
    		\caption{Energy profile developing in time}
    		\label{fig:init_energy2}
  	\end{subfigure}%
  	\caption{
  	Gaussian and Cubic kernels with approximation space $K=5$ and $I=20$ blocks on equidistant points after 10 periods
	}
  	\label{fig:init_approx2}
\end{figure} 
%
%
%

\subsection{Advection with Inflow Boundary Conditions}
In the following test from \cite{glaubitz2021towards}, we consider the advection equation \eqref{eq:linear} with $a=1$ in the domain $[0,1]$. 
The BC and IC  are
\begin{equation}\label{eq:BC_inflow_bump} 
	g(t)=u_{\mathrm{init}}(0.5-t), \quad 
  	u_{\mathrm{init}}(x) = 
  	\begin{cases}
    		\mathrm{e}^{8} \mathrm{e}^{ \frac{-8}{1-(4x-1)^2} }& \text{if } 0 < x < 0.5, \\ 
    		0 & \text{otherwise}. 
  	\end{cases}
\end{equation}
We have a smooth IC and an inflow BC at the left boundary $x=0$. We apply cubic splines with  constants as basis functions and the discretization 
	\begin{eq}\label{eq:lin_discr}
	\mathbf{u}_t  +a D \mathbf{u} = P^{-1} \mathbb{S}.
\end{eq} 
with the simultaneous approximate terms (SAT) 
$
	\mathbb{S} := [\mathbb{S}_0,0,\dots,0]^T, \quad 
	\mathbb{S}_0 := - (u_0-g).
$
In \cref{fig:rand_15} - \cref{fig:rand_20}, we show the solutions at time t = 0.5 with $K=5$ and $I=15, 20$ elements using equidistant point and randomly disturbed equidistant points. The numerical solutions using disturbed points in \cref{fig:rand_15} has wiggles but these are reduced by increasing the number of blocks, see \cref{fig:rand_20}.  Note that  that the wiggles  are more pronounced if the point selection is not distributed symmetrically around the midpoints, e.g. for the Halton points in \cref{fig:Halton_15} - \cref{fig:Halton_20}. Next, we focus on the error behavior. As mentioned before, the RBF methods can reach spectral accuracy for smooth solutions. In \cref{fig:init_approx_5},  the error behaviour for $K=3-7$ basis functions using 20 blocks is plotted in a logarithmic scale. Spectral accuracy is indicated by the (almost) constant slope.

\begin{figure}[tb]
	\centering 
	\begin{subfigure}[b]{0.33\textwidth}
		\includegraphics[width=\textwidth]{%
      		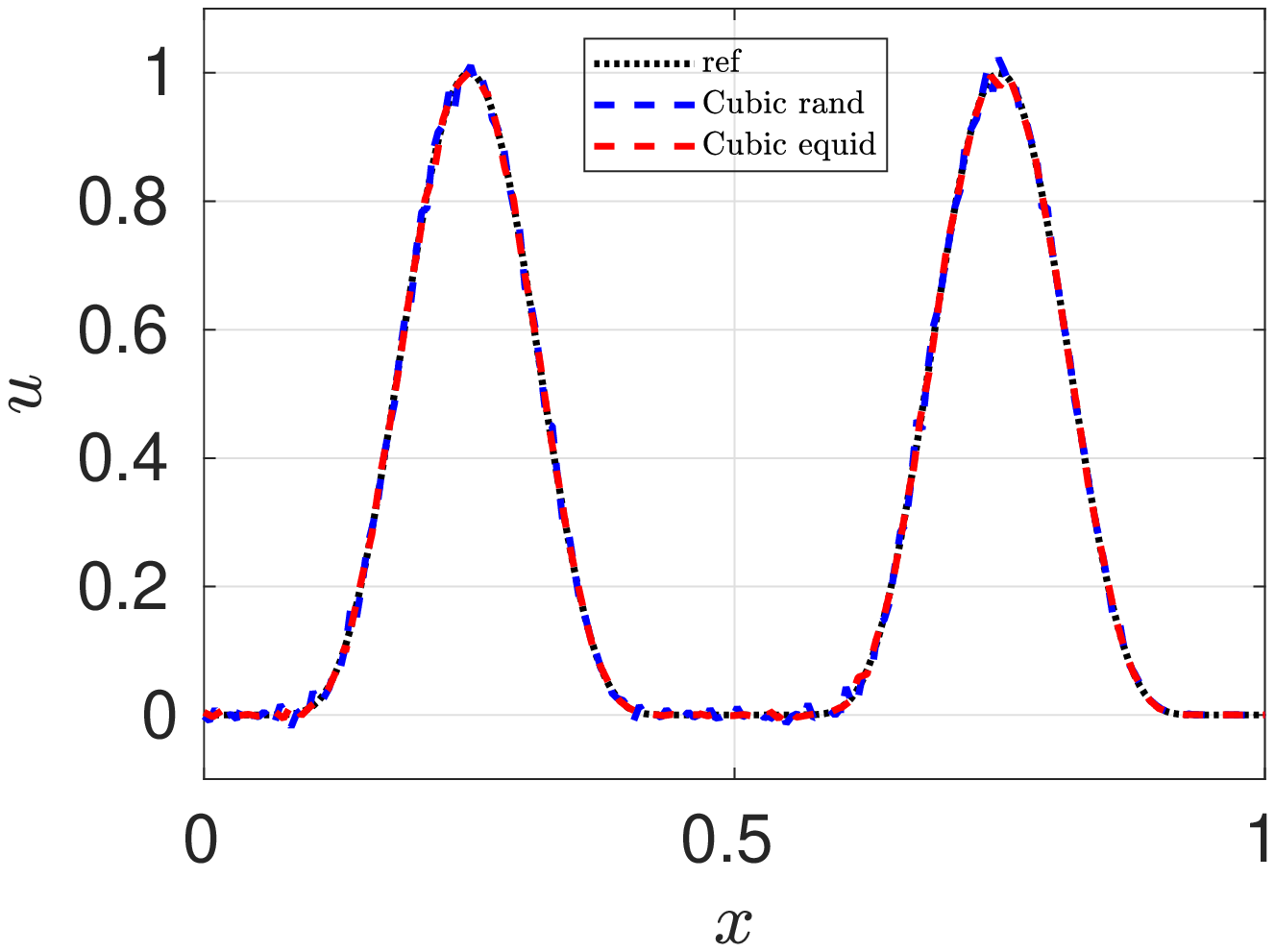} 
    		\caption{$I=15$ Blocks}
    		\label{fig:rand_15}
  	\end{subfigure}%
		~
  	\begin{subfigure}[b]{0.33\textwidth}
		\includegraphics[width=\textwidth]{%
      		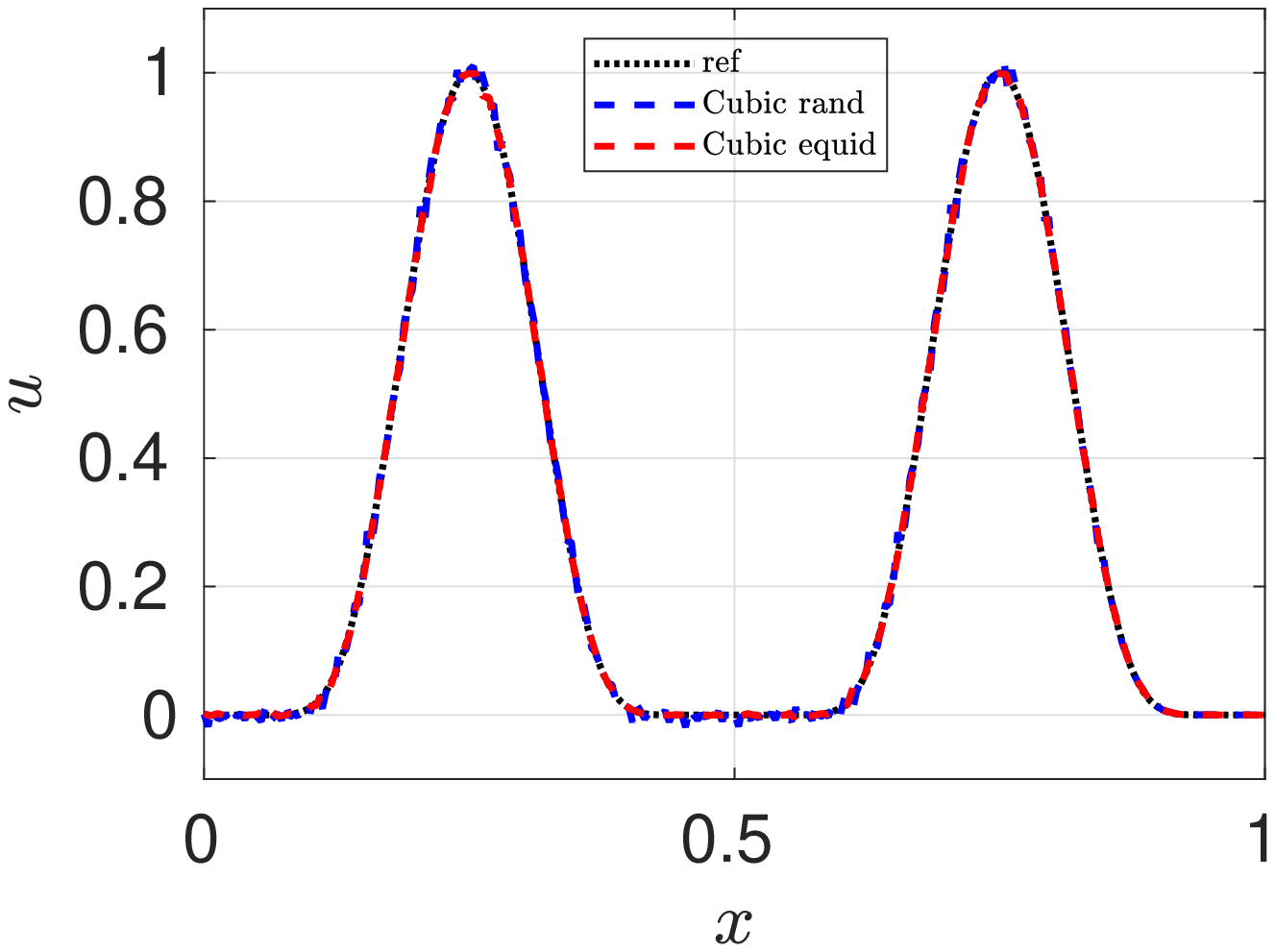} 
    		\caption{I=20 Blocks,}
    		\label{fig:rand_20}
  	\end{subfigure}%
		~\\
		\begin{subfigure}[b]{0.33\textwidth}
		\includegraphics[width=\textwidth]{%
      		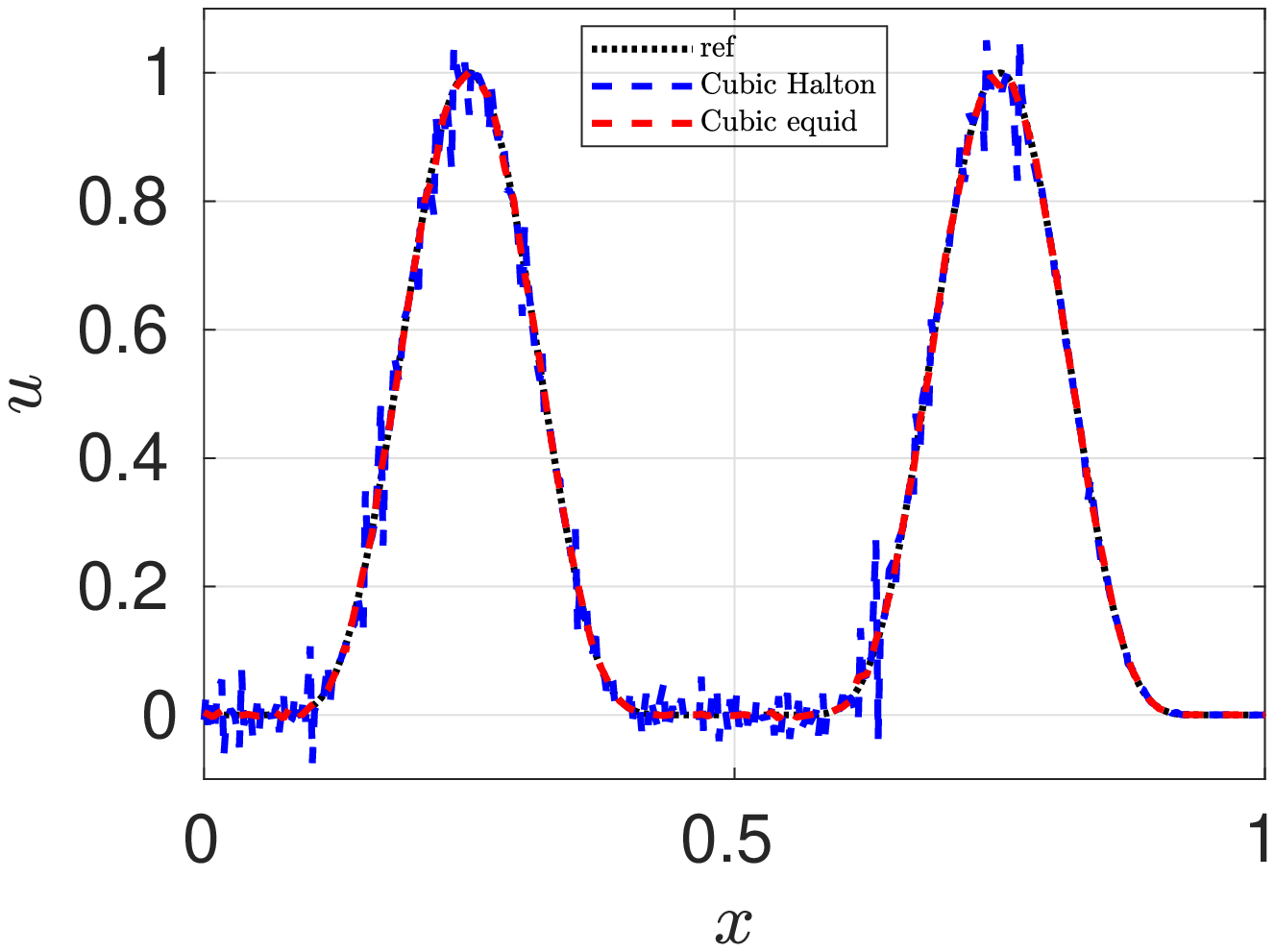} 
    		\caption{$I=15$ Blocks}
    		\label{fig:Halton_15}
  	\end{subfigure}%
		~
  	\begin{subfigure}[b]{0.33\textwidth}
		\includegraphics[width=\textwidth]{%
      		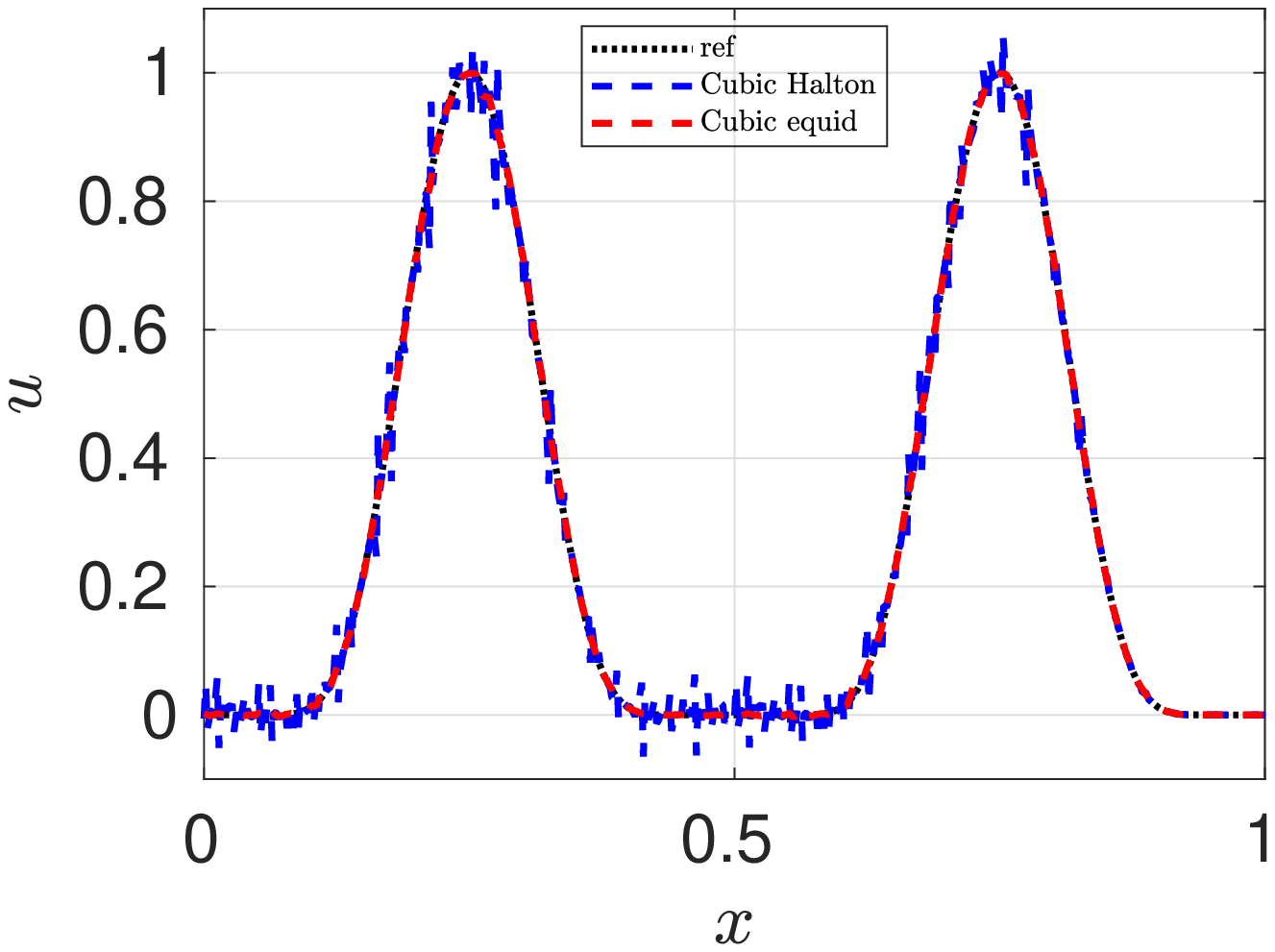} 
    		\caption{I=20 Blocks,}
    		\label{fig:Halton_20}
  	\end{subfigure}%
  	\caption{
  	Cubic kernel with approximation space $K=5$  on equidistant,  and Halton  points 	}
  	\label{fig:init_approx_4}
\end{figure}

\begin{figure}[tb]
	\centering 
		\begin{subfigure}[b]{0.33\textwidth}
		\includegraphics[width=\textwidth]{%
      		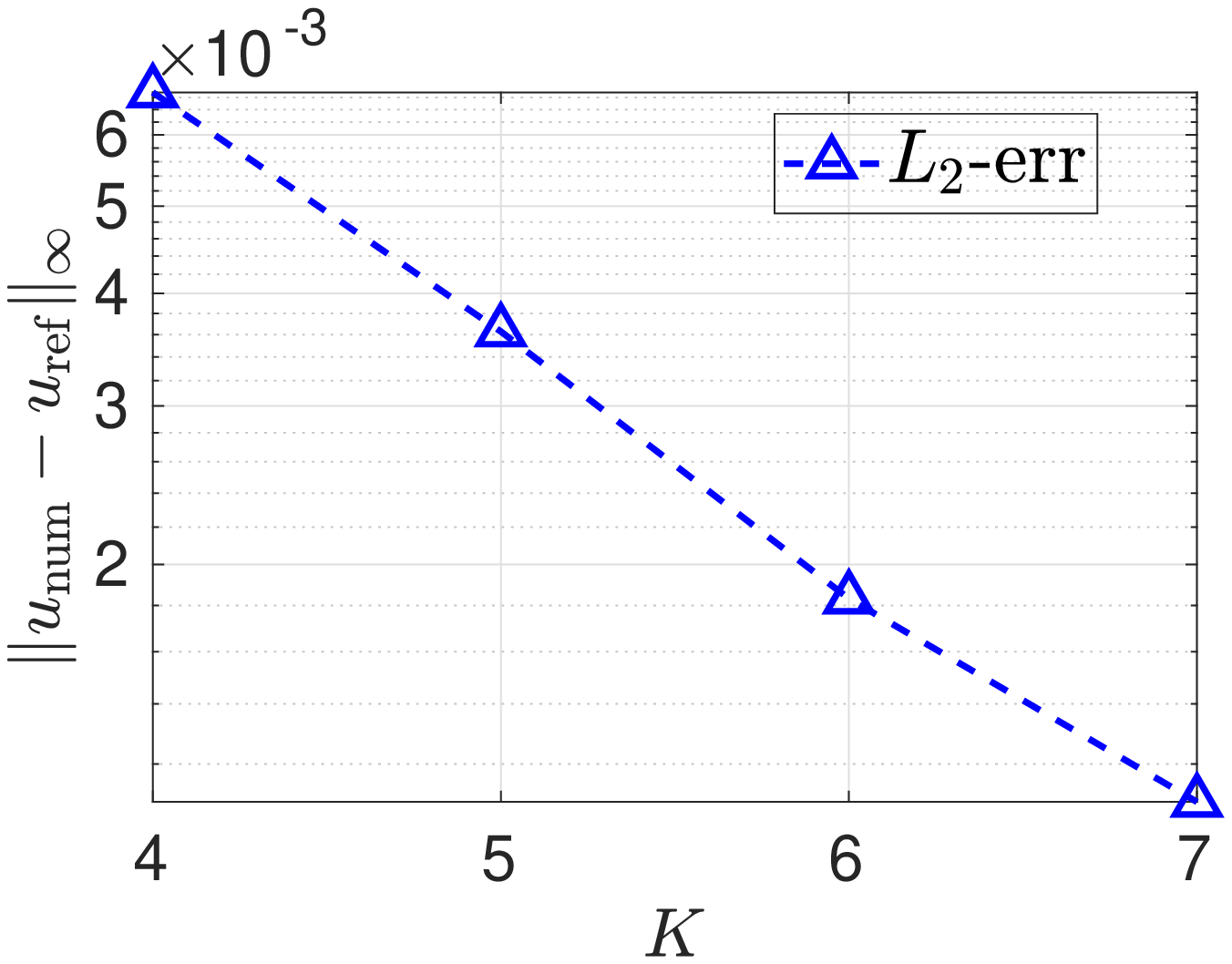} 
    		\caption{$L^2$-error}
    		\label{fig:l2_error}
  	\end{subfigure}%
		~
  	\begin{subfigure}[b]{0.33\textwidth}
		\includegraphics[width=\textwidth]{%
      		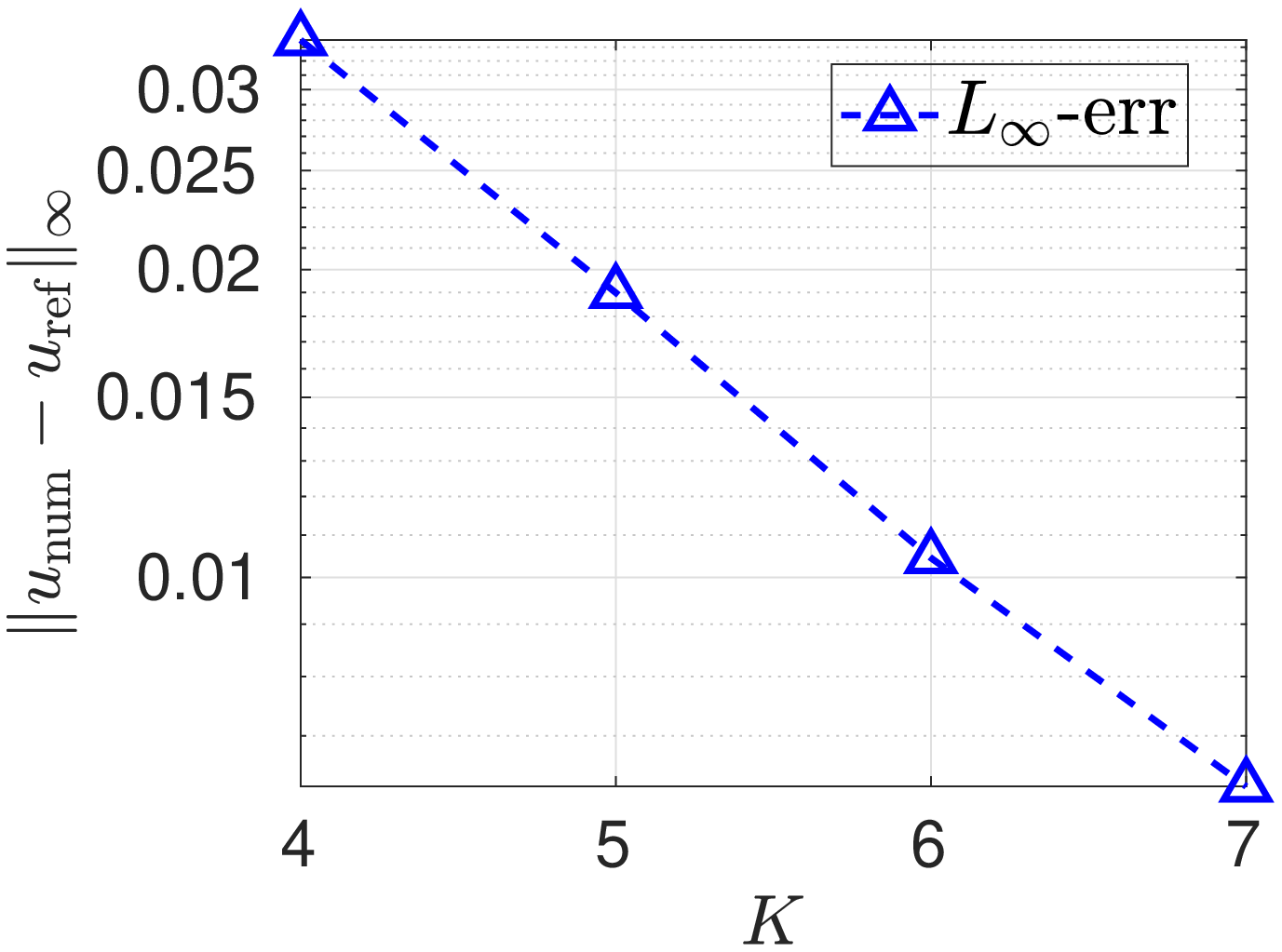} 
    		\caption{Maximum error }
    		\label{fig:linf_error}
  	\end{subfigure}%
  	\caption{Error plots using cubic kernels with approximation space $K=4-7$  on equidistant points with $I=20$ blocks.
	For $K=5$,  the errors correspond to the solutions  printed in the red dotted line on the right side of \cref{fig:init_approx_4}.}
  	\label{fig:init_approx_5}
\end{figure}

\subsection{Advection-Diffusion}
Next, the boundary layer problem from  \cite{yuan2006discontinuous} is considered
\begin{equation*}
	\partial_t u + \partial_x u = \kappa \partial_{xx}^2 u, \quad 0 \leq x \leq 0.5, \ t>0.\\
\end{equation*} 
The initial condition is $u(x,0)=2x$ and  the boundary conditions are $u(0,t)=0$ and 
$u(0.5,t)=1$. The exact steady state solution is  
$
u(x)= \frac{\exp \left(\frac{x}{\kappa} \right) -1}{\exp \left(\frac{1}{2\kappa} \right) -1}.
$
Cubic splines and Gaussian kernels with shape parameter $1$ are used together with constants.  We expect to obtain better results using Gaussian kernels due to structure of the steady state solution.
In \cref{fig:init_approx_diff_3}, we show the solutions for different times  using $K=5$ elements on equidistant grid points with diffusion parameters $\kappa=0.2$ and $\kappa=0.1$. 
\begin{figure}[tb]
	\centering 
  	\begin{subfigure}[b]{0.33\textwidth}
		\includegraphics[width=\textwidth]{%
      		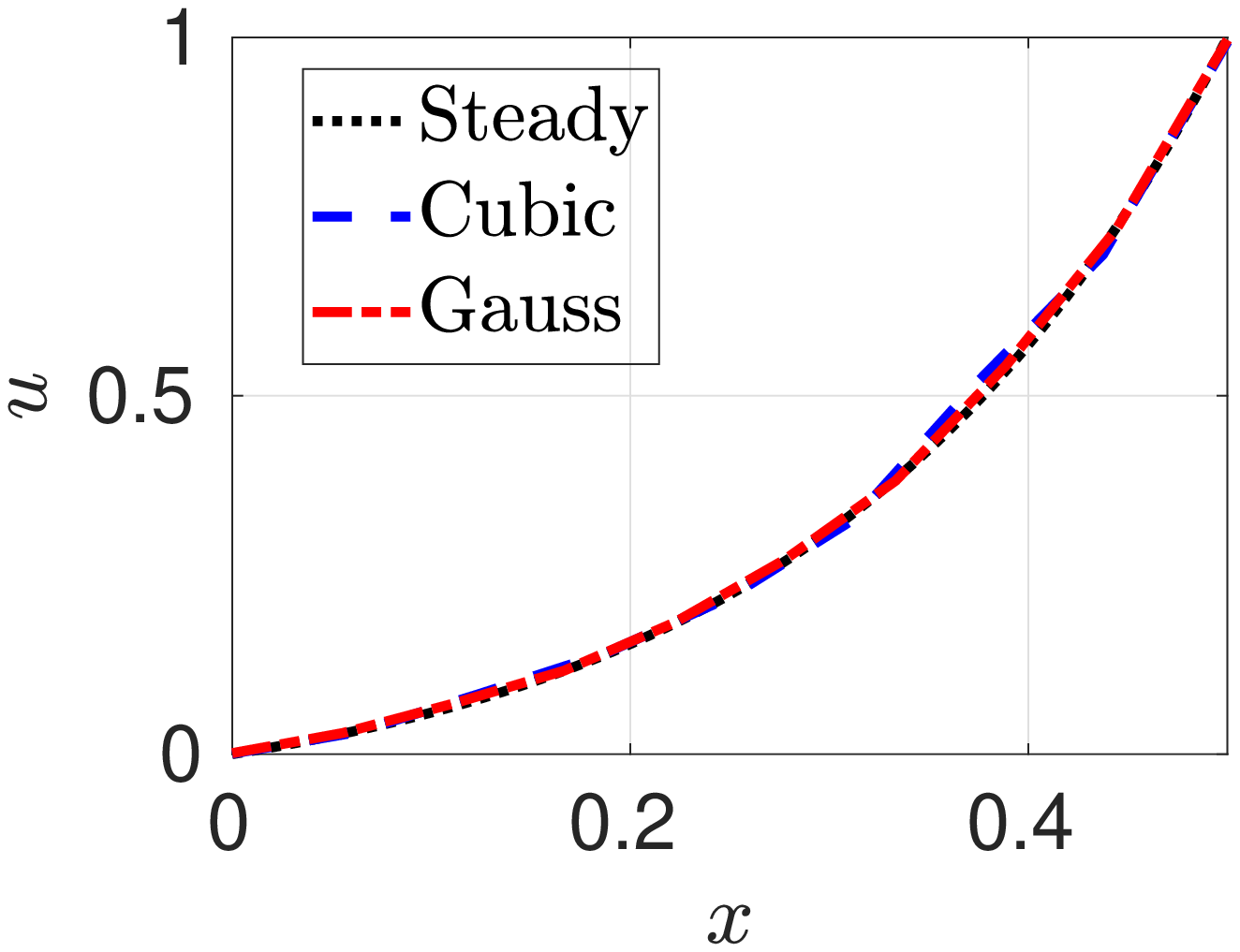} 
    		\caption{$\kappa=0.2$  }
    		\label{fig:kappa_02}
  	\end{subfigure}%
		~
  	\begin{subfigure}[b]{0.33\textwidth}
		\includegraphics[width=\textwidth]{%
      		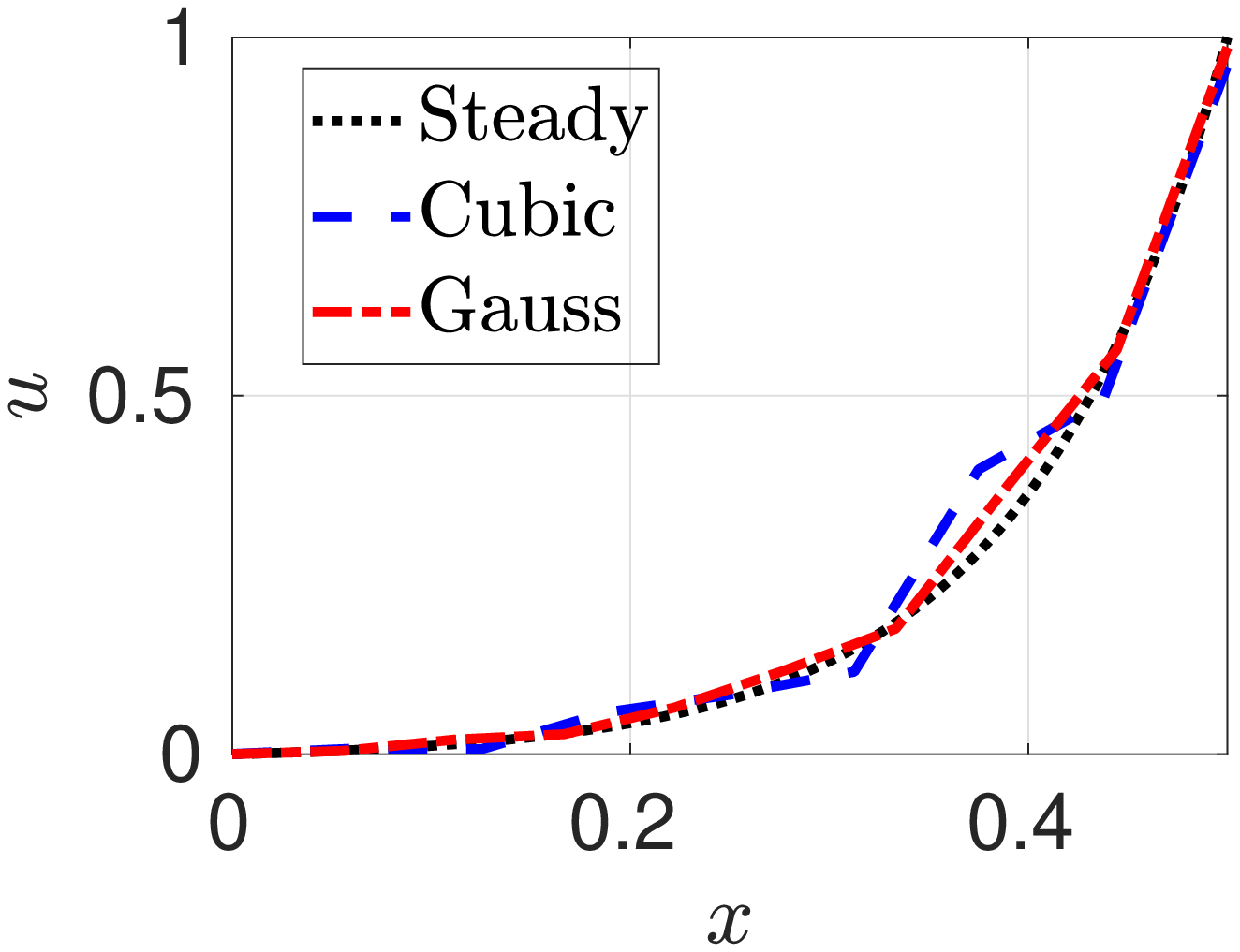} 
    		\caption{$\kappa=0.1$ }
    		\label{fig:kappa_01}
  	\end{subfigure}%
  	\caption{
  	Gaussian and Cubic kernels with approximation space $K=5$ and $I=1$ block on equidistant points  at $T=2$.
	}
  	\label{fig:init_approx_diff_3}
\end{figure} 
 Some overshoots can be seen in the more steep case for $\kappa=0.1$. This behavior can be circumvented by using more degrees of freedom and multi-blocks which are avoided in this case. 


\subsection{2D Linear Advection} 
\label{sub:num_2d} 
We conclude our examples with a 2D case and consider the linear advection equation:
\begin{equation}\label{eq:2d}
 \partial_t u(x,y,t) + a \partial_x u(x,y,t) +b  \partial_y u(x,y,t)=0
\end{equation}
with constants $a,b \in \R$. 
\subsubsection{Periodic Boundary Conditions}
In our first test, $a=b=1$ are used in \eqref{eq:2d}. The initial condition is $ u(x,y,0)= \mathrm{e}^{-20\left( (x-0.5)^2+(y-0.5)^2 \right)}$
for $(x,y) \in [0,1]^2$ and  periodic boundary conditions, i.\,e., $u(0,y,t)=u(1,y,t)$ and $u(x,0,t)=u(x,1,t)$, are considered. 
The coupling at the boundary was again done via SAT terms.  
We use cubic kernels ($K=13$) equipped with constants. 
\cref{fig:periodic_2} illustrates the numerical solution at time $T=1$. 
The bump has once left the domain at the right upper corner and entered again in the left lower corner. It reaches its initial position at $T=1$. 
No visible differences between the numerical solution at $T=1$ and the initial condition can be seen. 
In \cref{fig:periodic_3} the energy is reported over time. 
We notice a slide decrease of energy when the bump is leaving the domain (at $t=0.5$) due to weakly enforced slightly dissipative SBP-SAT coupling.

\begin{figure}[tb]
	\centering 
	\begin{subfigure}[b]{0.3\textwidth}
		\includegraphics[width=\textwidth]{%
      		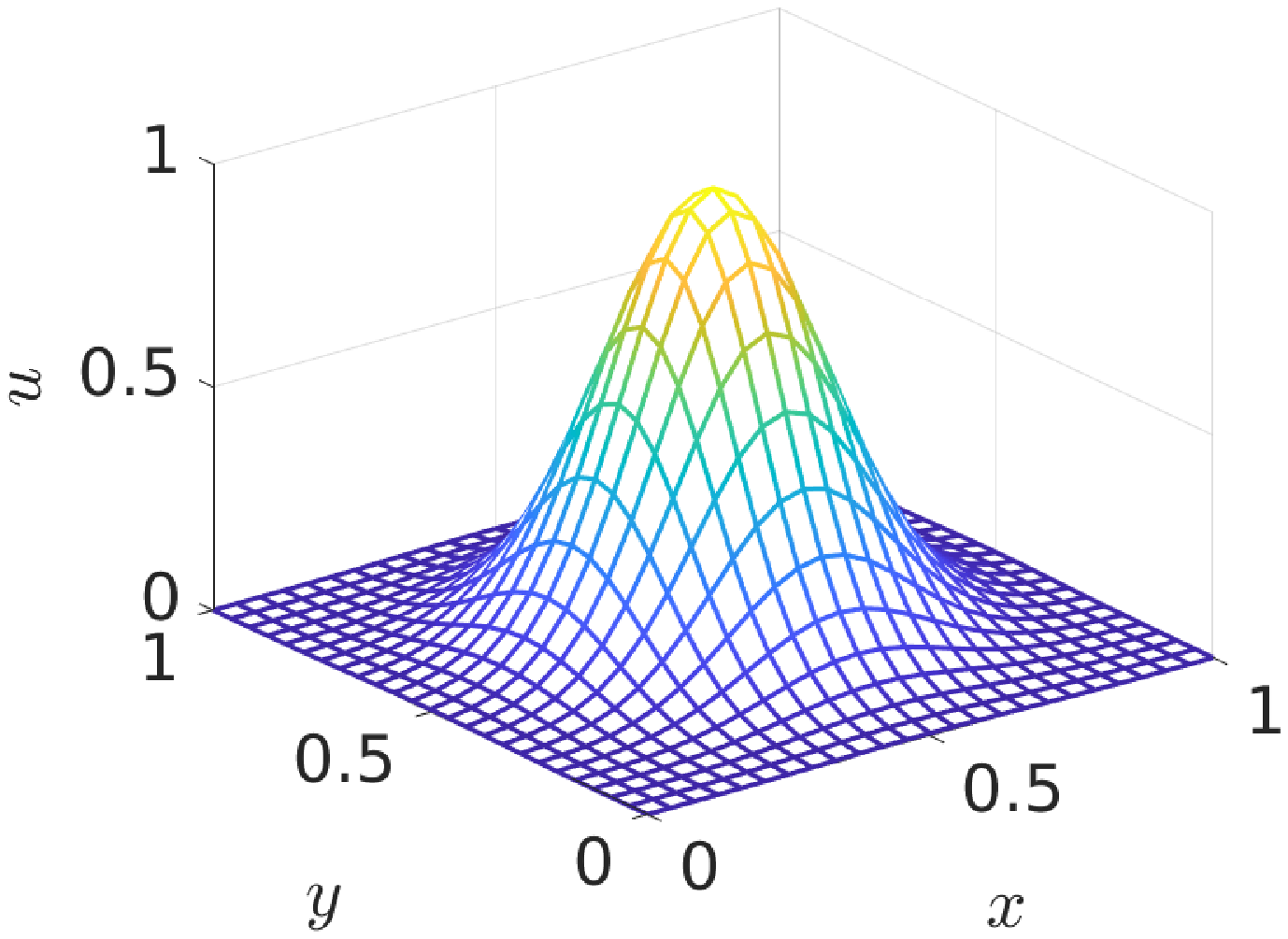} 
    		\caption{Initial Condition}
    		\label{fig:periodic_1}
  	\end{subfigure}%
		~
  	\begin{subfigure}[b]{0.3\textwidth}
		\includegraphics[width=\textwidth]{%
      		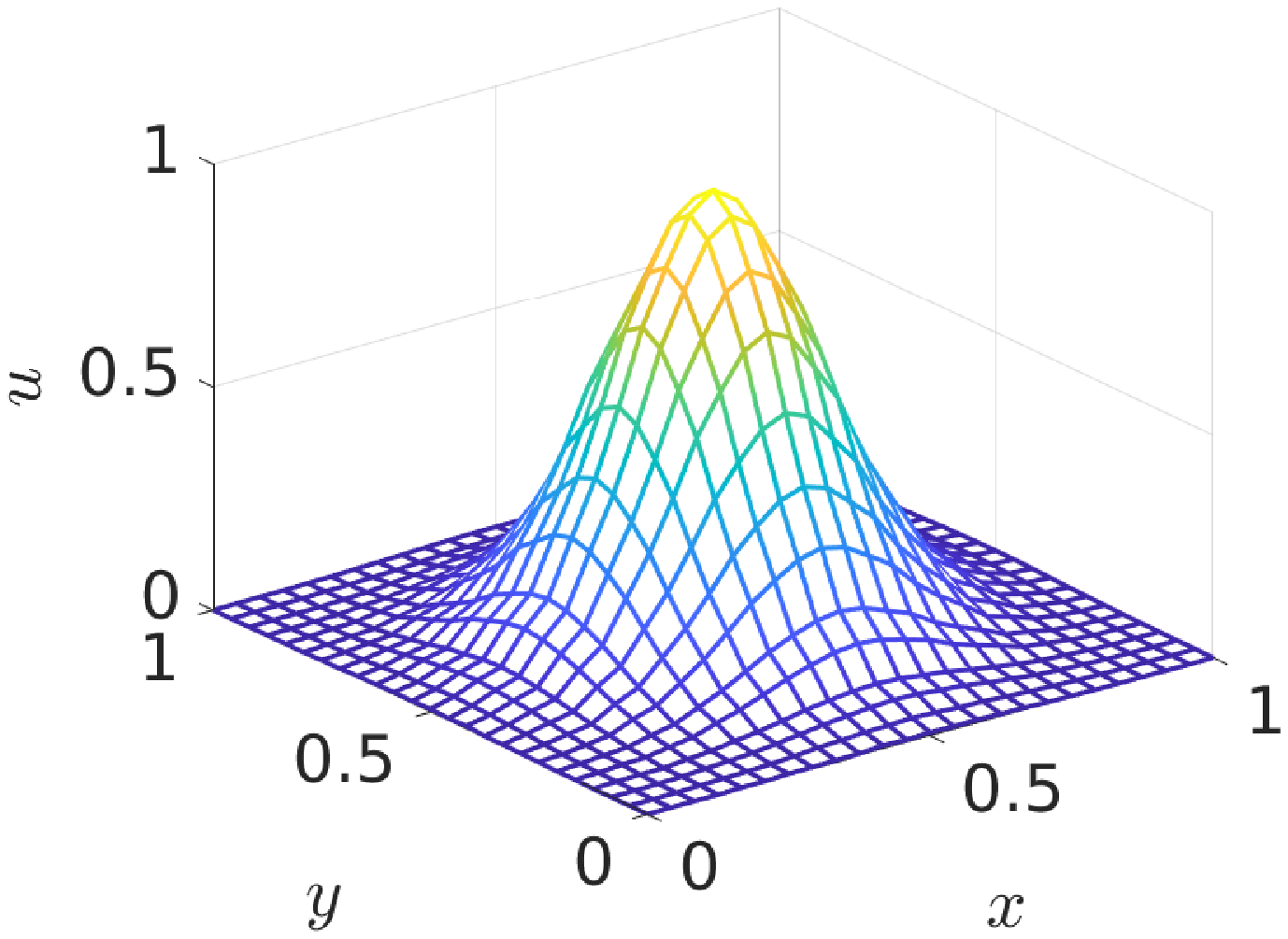} 
    		\caption{Numerical Solution $T=1$ }
    		\label{fig:periodic_2}
  	\end{subfigure}%
	~
  	\begin{subfigure}[b]{0.3\textwidth}
		\includegraphics[width=\textwidth]{%
      		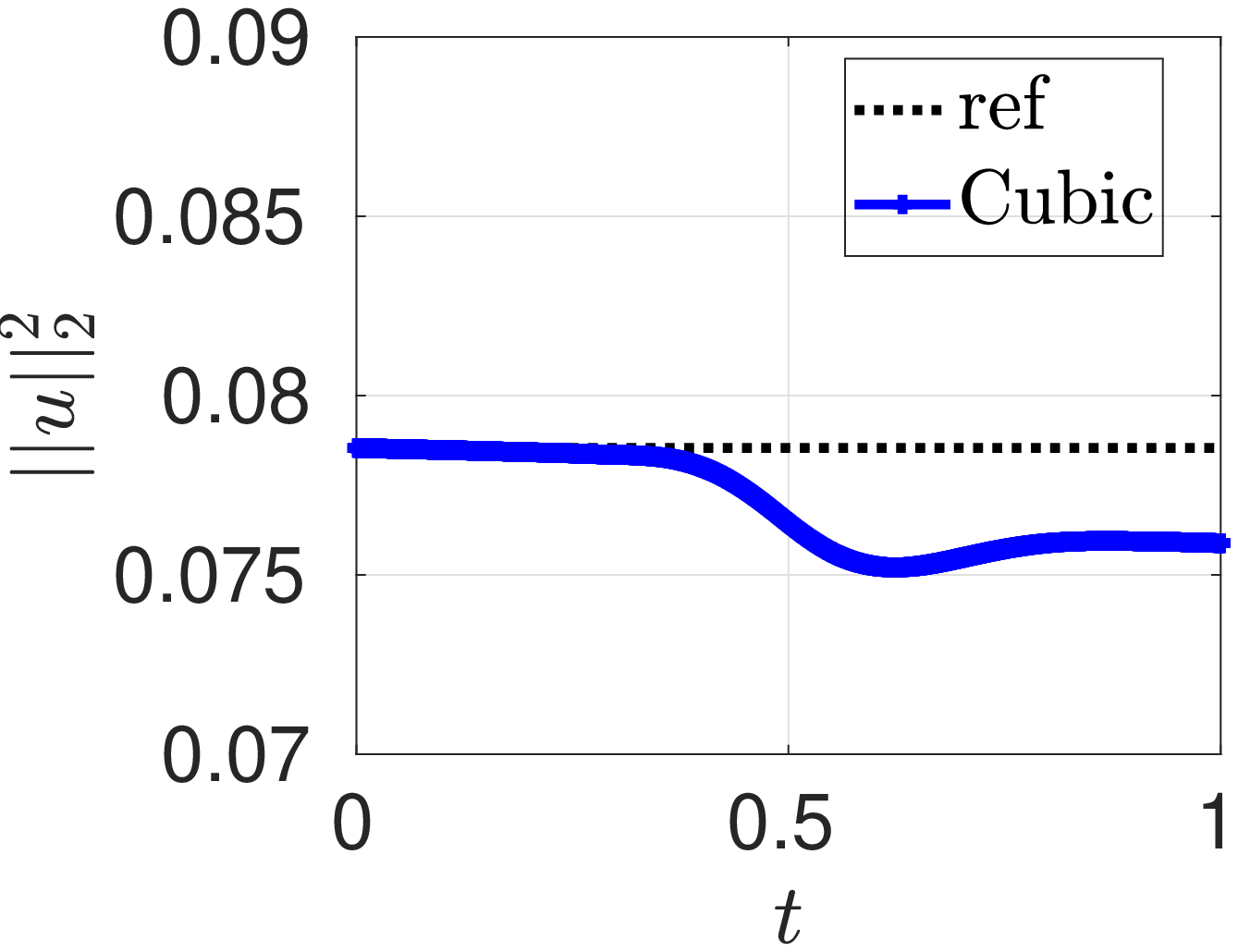} 
    		\caption{Energy over   $T$}
    		\label{fig:periodic_3}
  	\end{subfigure}%
  	\caption{
  	Cubic kernels with approximation space $K=13$  on equidistant points, 
	}
  	\label{fig:init_approx_2d_periodic}
\end{figure} 
\subsubsection{Inflow Conditions}
In the last simulation, we consider \eqref{eq:2d} with $a=0.5$, $b=1$, initial condition 
  $u(x,y,0)= \mathrm{e}^{-20\left( (x-0.25)^2+(y-0.25)^2 \right)}$ for $(x,y) \in [0,1]^2$ and zero inflow 
 $u(0,y,t)=0$, and  $u(x,0,t)=0$.
We  again use cubic kernels ($K=13$) equipped with constants.  The boundary conditions are enforced weakly via SAT terms. 
 The initial condition lies in the left corner, cf. \cref{fig:periodic_1}. In \cref{fig:inflow_2}, the numerical solution is shown.
 The bump moves in $y$ direction with speed one and in $x$-direction with speed $0.5$. 
\cref{fig:inflow_3} shows  a slight decrease of the  energy  over time due the bump leaving the domain.
\begin{figure}[tb]
	\centering 
	\begin{subfigure}[b]{0.3\textwidth}
		\includegraphics[width=\textwidth]{%
      		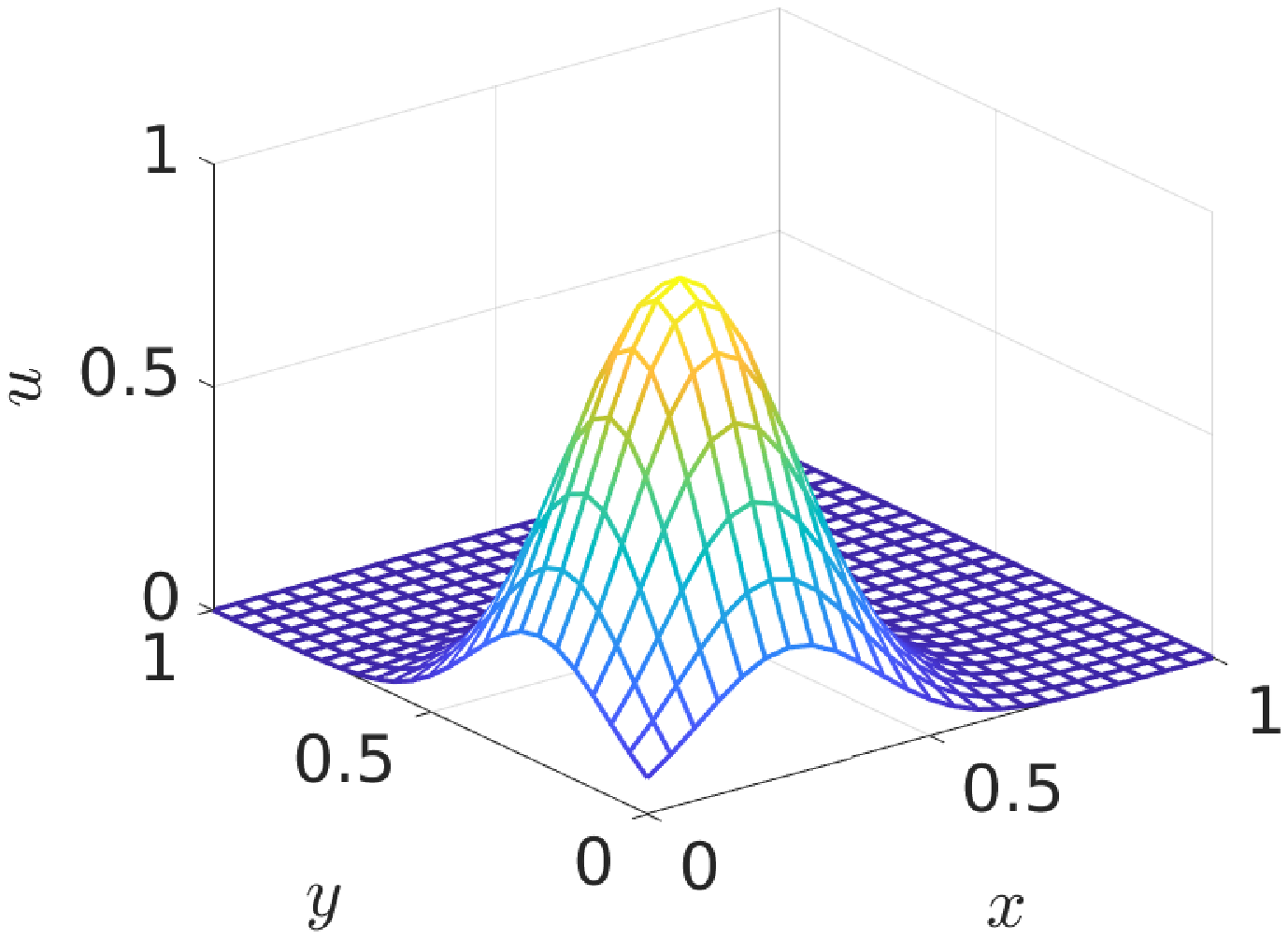} 
    		\caption{Initial Condition}
    		\label{fig:inflow_1}
  	\end{subfigure}%
		~
  	\begin{subfigure}[b]{0.3\textwidth}
		\includegraphics[width=\textwidth]{%
      		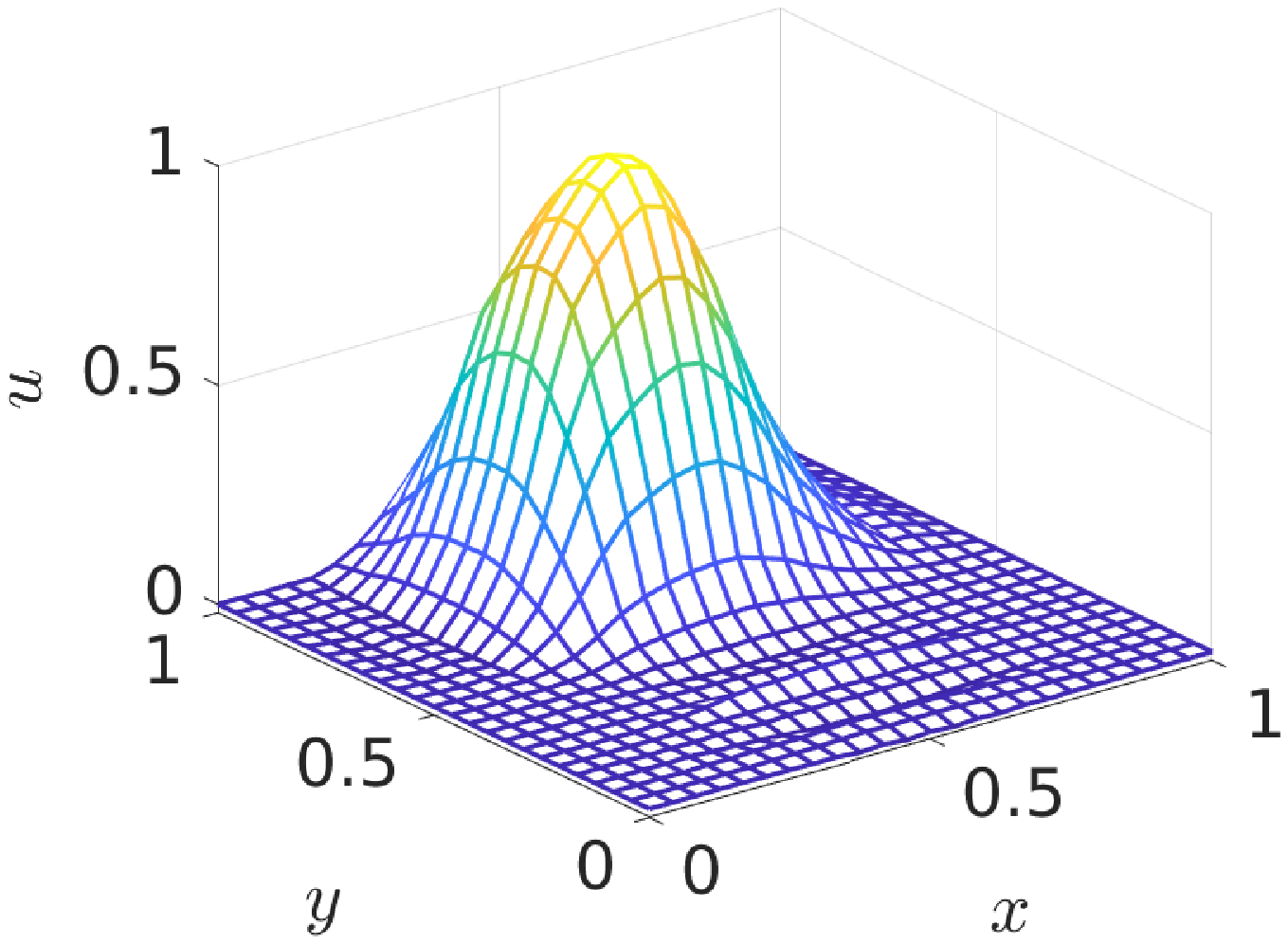} 
    		\caption{Numerical Solution $T=0.5$ }
    		\label{fig:inflow_2}
  	\end{subfigure}%
	~
  	\begin{subfigure}[b]{0.3\textwidth}
		\includegraphics[width=\textwidth]{%
      		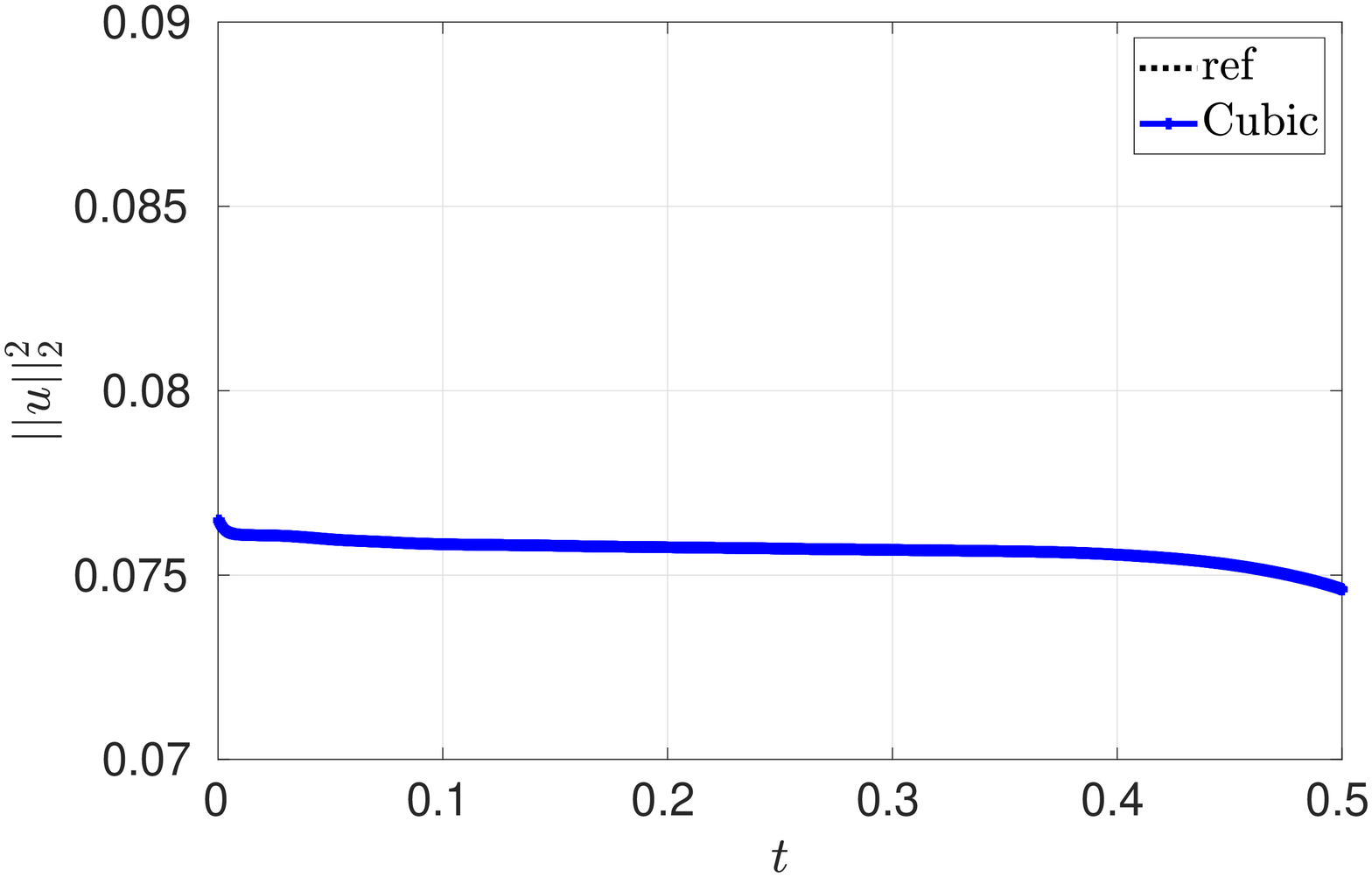} 
    		\caption{Energy over $T$}
    		\label{fig:inflow_3}
  	\end{subfigure}%
  	\caption{
  	Cubic kernels with approximation space $K=13$  on equidistant points 
	}
  	\label{fig:init_approx_2d_inflow}
\end{figure}

%% file: summary.tex
\section{Concluding Thoughts} 
\label{sec:summary} 

RBF methods are a popular tool for numerical PDEs. However, despite their success for problems with sufficient inherent dissipation, stability issues are often observed for advection-dominated problems. 
In this work, we used the FSBP theory combined with a weak enforcement of BCs to develop provable energy-stable RBF methods.  
We found that one can construct RBFSBP operators 
by using oversampling to obtain suitable positive quadrature formulas. 
Existing RBF methods do not satisfy such an RBFSBP property, either because they are based on collocation or because an inappropriate quadrature is used.
Our findings imply that FSBP theory provide a building block for systematically developing stable RBF methods, filling a critical gap in the RBF theory. 
The focus in this paper was on global RBF methods, future works will address the extension to local RBF methods.

%% file: app_polynomial.tex
\section{Necessity of Polynomials in RBFs}\label{app_polynomial}

For completeness, we shortly explain why the RBF interpolant \cref{eq:RBF-interpol} exists uniquely when the kernel $\varphi$ is conditionally positive definite of order $m$ and polynomials of degree up to $m-1$ are incorporated. 
To this end, recall that $\varphi$ is \emph{conditionally positive definite of order $m$} when 
	\begin{equation} 
		\boldsymbol{\alpha}^T \Phi \boldsymbol{\alpha} > 0 
	\end{equation} 
	for all $\boldsymbol{\alpha} \in \R^K\setminus\{\mathbf{0}\}$ that satisfy \cref{eq:cond2}, where $\Phi$ is given by \cref{eq:Phi_P}. 
	Further, \cref{eq:cond2} is equivalent to $P^T \boldsymbol{\alpha} = \mathbf{0}$.  
	Next note that the RBF interpolant \cref{eq:RBF-interpol} exists uniquely if and only if the linear system \cref{eq:system} has a unique solution for every $\mathbf{u}$, which is equivalent to the corresponding homogeneous linear system 
	\begin{align} 
		\Phi \boldsymbol{\alpha} + P \boldsymbol{\beta} & = \mathbf{0}, \label{eq:system2a} \\ 
		P^T \boldsymbol{\alpha} & = \mathbf{0}, \label{eq:system2b}
	\end{align} 
	admitting only the trivial solution, $\boldsymbol{\alpha} = \mathbf{0}$ and $\boldsymbol{\beta} = \mathbf{0}$. 
	To show that this is the case, we multiply both sides of \cref{eq:system2a} by $\boldsymbol{\alpha}^T$ from the left, which yields 
	\begin{equation}\label{eq:system3a}
		\boldsymbol{\alpha}^T \Phi \boldsymbol{\alpha} = \mathbf{0}, 
	\end{equation} 
	since $\boldsymbol{\alpha}^T P = \mathbf{0}^T$ due to \cref{eq:system2b}. 
	Further, for conditionally positive definite $\varphi$, \cref{eq:system3a} implies $\boldsymbol{\alpha} = \mathbf{0}$. 
	Substituting $\boldsymbol{\alpha} = \mathbf{0}$ into \cref{eq:system2a} yields $P \boldsymbol{\beta} = \mathbf{0}$, which means that the polynomial 
	\begin{equation} 
		p(x) = \sum_{l=1}^m \beta_l p_l(x) 
	\end{equation} 
	has zeros $x_1,\dots,x_K$. 
	Finally, for $m \leq K$, this can only be the case if $\boldsymbol{\beta} = \mathbf{0}$.

%% file: app_nondiagonal.tex
\section{RBFSBP property with Non-diagonal Norm Matrix} 
\label{sec:existing_nondiagonal} 
In \cref{sec:existing}, we  demonstrated that there exist no diagonal $P$ such that the RBFSBP properties are fulfilled in general.  
In the general definition  \eqref{def:SBP_RBF}, $P$ must only be symmetric positiv definite and not necessarily 
diagonal. Therefore,  some non-diagonal norm matrix might exists fulfilling the RBFSBP property. Here, we demonstrate that this is  not the case. 
To investigate this, we set $X_K=Y_N$. The differentiation operators ${D \in \R^{N \times N}}$ of classic global RBF methods are usually constructed to be exact for the elements of the finite dimensional function space $\mathcal{R}_{m}(Y_N)$. 
Unfortunately, neither the norm matrix $P$ nor the matrix $Q$ are explicitly part of RBF methods, which only come with an RBF-exact differentiation operator ${D \in \R^{N \times N}}$ for the cardinal functions. 
That said, we will now  demonstrate that in many cases existing collocation RBF methods cannot satisfy the RBFSBP property since certain conditions are violated. \\
To this end, let ${D \in \R^{N \times N}}$ be the RBF-differentiation operator. 
 We assume that there exist a positive definite  and symmetric norm matrix $P \in \R^{N \times N}$ and a matrix $Q \in \R^{N \times N}$ such that (see \cref{def:SBP_RBF}) 
\begin{equation}\label{eq:SBP_conds}
	D = P^{-1} Q, \quad 
	Q + Q^T = B.
\end{equation} 
The two conditions in \cref{eq:SBP_conds} can be combined to 
\begin{equation} \label{eq_SBP_cond_2}
	PD + (PD)^T - B = 0.
\end{equation}
Next, we assume that the RBF interpolant include polynomials of most degree  $m-1 \geq 0$. 
In this case, $\mathcal{R}_{1}(Y_N)$ contains constants and 
$P$ must be associated with a $\mathcal{R}_{1}(Y_N)$-exact quadrature formula. 
Since $D$ is $\mathcal{R}_{1}(Y_N)$-exact, this can be reformulated as 
\begin{equation} \label{eq_optimitzaion}	\int_{x_L}^{x_R} 1 \partial_x f d x= f|_{x_L}^{x_R} \Longleftrightarrow 	\mathbf{1}^T P D \mathbf{f}  = f|_{x_L}^{x_R}  \quad \forall f \in \mathcal{R}_{1}(Y_N)
\end{equation}
Since $D$ and $\mathbf{f}$ are formulated with respect to the same basis $\Span \{c_k\}$. The entries of  $D$ are given by
$
D_{jk} =c_k'(x_j)
$
with collocation points $x_j$. Hence,  \eqref{eq_optimitzaion} is used for 
every basis element $\mathrm{span} \{ c_i\}$, e.g. for $c_1$:
\begin{equation}\label{eq_conditions_3}
  c_1|_{x_L}^{x_R}=c_1(x_R)-c_1(x_L)=\int_{x_L}^{x_R} 1 \partial_x c_1 d x= \sum_{j=1}^N w_j c'_1(x_j)=\mathbf{1}^T P \mathbf{c_1'}.
\end{equation}
Since $\mathbf{c_i'}$ are the columns of the derivative matrix.
 We can collect every basis element using  \eqref{eq_conditions_3} resulting in  
\begin{equation} \label{eq_optimitzaion_2}
	\mathbf{1}^T P D  = 	 \mathbf{m}.
\end{equation}
with $\mathbf{m} = [ c_1|_{x_L}^{x_R}, \dots, c_N|_{x_L}^{x_R} ]$. 
We shall now summarize the above discussion: 
Let $m-1 \geq 0$ and ${b(P):= \| PD + (PD)^T - B \|_2}$. 
Moreover, for given $D$ let us consider the following optimization problem: 
\begin{equation}\label{eq:opt_problem}
	\min_{P} \left\{ \, b(P) \ \text{ s.t. } \ P = P^T, \ P > 0, \ \mathbf{1}^T P D = \mathbf{m} \, \right\}
\end{equation}
If the differentiation operator $D$ of a classic global RBF method satisfies the FSBP property, then minimizers $P^*$ of the optimization problem \cref{eq:opt_problem} satisfy $b(P^*) = 0$. 
There exist a suitable quadrature formula to determine $P$ through the  minimization problem \eqref{eq:opt_problem}.
It should be stressed that $b(P) = 0$ is necessary for the given $D$ to satisfy the SBP property, but not sufficient. This follows directly from \cite[Lemma 4.3]{glaubitz2022nonpolnyomial} containing the fact that the derivatives of the basis functions are integrated exactly. 
\begin{table}[tb]
\renewcommand{\arraystretch}{0.8}
\centering 
  	\begin{tabular}{c c c c c c c c c} 
	\toprule 
	\multicolumn{9}{c}{Equidistant points} \\ \hline 
	& & \multicolumn{3}{c}{cubic} & & \multicolumn{3}{c}{quintic} \\ \hline
	$N$/$m-1$ & & $0$ & $1$ & $2$ & & $0$ & $1$ & $2$ \\ \hline 
	$10$ & 
		& $1.5e-10$ & $8.1e-11$ & $5.8e-10$ &
		& $5.8e-10$ & $4.0e-10$ & $4.0e-10$ \\ 
	$20$ & 
		& $3.8e-10$ & $9.3e-10$ & $5.5e-10$ &
		& $1.0e-05$ & $7.5e-08$ & $6.0e-08$ \\  \hline \\
	\end{tabular} 
	\\ 
	\begin{tabular}{c c c c c c c c c} 
	\multicolumn{9}{c}{Halton points} \\ \hline 
	& & \multicolumn{3}{c}{cubic} & & \multicolumn{3}{c}{quintic} \\ \hline
	$N$/$m-1$ & & $-1$ & $0$ & $1$ & & $-1$ & $0$ & $1$ \\ \hline 
	$10$ & 
		& $1.8e-01$ & $2.0e-01$ & $2.0e-01$ &
		& $3.1e-01$ & $1.9e-01$ & $1.8e-01$ \\  
	$20$ & 
		& $1.9e-01$ & $2.0e-01$ & $1.9e-01$ &
		& $1.8e-01$ & $1.6e-01$ & $1.5e-01$ \\  \hline \\
	\end{tabular} \\ 
	\begin{tabular}{c c c c c c c c c} 
	\multicolumn{9}{c}{Random points} \\ \hline 
	& & \multicolumn{3}{c}{cubic} & & \multicolumn{3}{c}{quintic} \\ \hline
	$N$/$m-1$ & & $-1$ & $0$ & $1$ & & $-1$ & $0$ & $1$ \\ \hline 
	$10$ & 
		& $1.2e-01$ & $9.3e-02$ & $1.0e-01$ &
		& $9.3e-01$ & $5.9e-01$ & $5.7e-01$ \\  
	$20$ & 
		& $6.9e-02$ & $7.1e-02$ & $6.4e-02$ &
		& $3.2e+02$ & $5.1e+00$ & $3.6e-01$ \\  
	\bottomrule
	\end{tabular}
\caption{Residual $\| PD + (PD)^T - B \|_2$ for the determined norm matrix $P$ in the case of equidistant, Halton, and random grid points}
\label{tab:SBP_existing}
\end{table}
In our implementation we solved \cref{eq:opt_problem} using Matlab's \emph{CVX} \cite{grant2014cvx}. 
The results for different numbers and types of grid points $\mathbf{x}$ as well as kernels $\varphi$ and polynomial degrees $m-1$ can be found in \cref{tab:SBP_existing}. 
Our numerical findings indicate that in all cases classic global RBF methods do \emph{not} satisfy the RBFSBP property. 
This can be noted from the residual $b(P) = \| PD + (DP)^T - B \|_2$ corresponding to the minimizer $P$ of \cref{eq:opt_problem} to be distinctly different from zero (machine precision  in our implementation is around $10^{-16}$). 
This result is not suprising and in accordance with the observations 
made in the literature \cite{glaubitz2021stabilizing, tominec2021stability}.

%% file: main_arxiv.bbl
\begin{thebibliography}{10}

\bibitem{abgrall2020analysis}
{\sc R.~{Abgrall}, J.~{Nordstr\"om}, P.~{\"Offner}, and S.~{Tokareva}}, {\em
  {Analysis of the {SBP-SAT} stabilization for finite element methods {I}:
  {L}inear problems}}, {J. Sci. Comput.}, 85 (2020), p.~28,
  \url{https://doi.org/10.1007/s10915-020-01349-z}.
\newblock Id/No 43.

\bibitem{abgrall2021analysis}
{\sc R.~Abgrall, J.~Nordstr{\"o}m, P.~{\"O}ffner, and S.~Tokareva}, {\em
  Analysis of the {SBP-SAT} stabilization for finite element methods part ii:
  Entropy stability}, Communications on Applied Mathematics and Computation,
  (2021), pp.~1--23.

\bibitem{carpenter2010revisting}
{\sc M.~H. {Carpenter}, J.~{Nordstr\"om}, and D.~{Gottlieb}}, {\em {Revisiting
  and extending interface penalties for multi-domain summation-by-parts
  operators}}, {J. Sci. Comput.}, 45 (2010), pp.~118--150,
  \url{https://doi.org/10.1007/s10915-009-9301-5}.

\bibitem{chan2019efficient}
{\sc J.~{Chan}, D.~C. {Del Rey Fern\'andez}, and M.~H. {Carpenter}}, {\em
  {Efficient entropy stable Gauss collocation methods}}, {SIAM J. Sci.
  Comput.}, 41 (2019), pp.~a2938--a2966,
  \url{https://doi.org/10.1137/18M1209234}.

\bibitem{cuomo2020greeks}
{\sc S.~{Cuomo}, F.~{Sica}, and G.~{Toraldo}}, {\em {Greeks computation in the
  option pricing problem by means of RBF-PU methods}}, {J. Comput. Appl.
  Math.}, 376 (2020), p.~14, \url{https://doi.org/10.1016/j.cam.2020.112882}.
\newblock Id/No 112882.

\bibitem{dehghan2017numerical}
{\sc M.~Dehghan and V.~Mohammadi}, {\em A numerical scheme based on radial
  basis function finite difference ({RBF-FD}) technique for solving the
  high-dimensional nonlinear {S}chr{\"o}dinger equations using an explicit time
  discretization: {R}unge--{K}utta method}, Computer Physics Communications,
  217 (2017), pp.~23--34.

\bibitem{fernandez2014review}
{\sc D.~C. {Del Rey Fern\'andez}, J.~E. {Hicken}, and D.~W. {Zingg}}, {\em
  {Review of summation-by-parts operators with simultaneous approximation terms
  for the numerical solution of partial differential equations}}, {Comput.
  Fluids}, 95 (2014), pp.~171--196,
  \url{https://doi.org/10.1016/j.compfluid.2014.02.016}.

\bibitem{fasshauer2007meshfree}
{\sc G.~E. Fasshauer}, {\em Meshfree Approximation Methods with MATLAB},
  vol.~6, World Scientific, 2007.

\bibitem{flyer2016role}
{\sc N.~{Flyer}, B.~{Fornberg}, V.~{Bayona}, and G.~A. {Barnett}}, {\em {On the
  role of polynomials in RBF-FD approximations. I: Interpolation and
  accuracy}}, {J. Comput. Phys.}, 321 (2016), pp.~21--38,
  \url{https://doi.org/10.1016/j.jcp.2016.05.026}.

\bibitem{flyer2012guide}
{\sc N.~{Flyer}, E.~{Lehto}, S.~{Blaise}, G.~B. {Wright}, and A.~{St-Cyr}},
  {\em {A guide to RBF-generated finite differences for nonlinear transport:
  shallow water simulations on a sphere}}, {J. Comput. Phys.}, 231 (2012),
  pp.~4078--4095, \url{https://doi.org/10.1016/j.jcp.2012.01.028}.

\bibitem{fornberg2005accuracy}
{\sc B.~Fornberg and N.~Flyer}, {\em Accuracy of radial basis function
  interpolation and derivative approximations on {1-D} infinite grids},
  Advances in Computational Mathematics, 23 (2005), pp.~5--20.

\bibitem{fornberg2015primer}
{\sc B.~Fornberg and N.~Flyer}, {\em A Primer on Radial Basis Functions With
  Applications to the Geosciences}, SIAM, 2015.

\bibitem{fornberg2011stable}
{\sc B.~Fornberg, E.~Larsson, and N.~Flyer}, {\em Stable computations with
  gaussian radial basis functions}, SIAM Journal on Scientific Computing, 33
  (2011), pp.~869--892.

\bibitem{gassner2016split}
{\sc G.~J. {Gassner}, A.~R. {Winters}, and D.~A. {Kopriva}}, {\em {Split form
  nodal discontinuous Galerkin schemes with summation-by-parts property for the
  compressible Euler equations}}, {J. Comput. Phys.}, 327 (2016), pp.~39--66,
  \url{https://doi.org/10.1016/j.jcp.2016.09.013}.

\bibitem{glaubitz2020shock}
{\sc J.~Glaubitz}, {\em Shock Capturing and High-Order Methods for Hyperbolic
  Conservation Laws}, Logos Verlag Berlin GmbH, 2020.

\bibitem{glaubitz2020stableQF}
{\sc J.~{Glaubitz}}, {\em {Stable high order quadrature rules for scattered
  data and general weight functions}}, {SIAM J. Numer. Anal.}, 58 (2020),
  pp.~2144--2164, \url{https://doi.org/10.1137/19M1257901}.

\bibitem{glaubitz2021constructing}
{\sc J.~Glaubitz}, {\em Construction and application of provable positive and
  exact cubature formulas}, IMA Journal of Numerical Analysis,  (2022),
  \url{https://arxiv.org/abs/2108.02848}.
\newblock To appear.

\bibitem{glaubitz2021stabilizing}
{\sc J.~{Glaubitz} and A.~{Gelb}}, {\em {Stabilizing radial basis function
  methods for conservation laws using weakly enforced boundary conditions}},
  {J. Sci. Comput.}, 87 (2021), p.~29,
  \url{https://doi.org/10.1007/s10915-021-01453-8}.
\newblock Id/No 40.

\bibitem{glaubitz2021towards}
{\sc J.~{Glaubitz}, E.~{Le Meledo}, and P.~{\"Offner}}, {\em {Towards stable
  radial basis function methods for linear advection problems}}, {Comput. Math.
  Appl.}, 85 (2021), pp.~84--97,
  \url{https://doi.org/10.1016/j.camwa.2021.01.012}.

\bibitem{glaubitz2022nonpolnyomial}
{\sc J.~Glaubitz, J.~Nordstr{\"o}m, and P.~{\"O}ffner}, {\em Summation-by-parts
  operators for general approximation spaces}, arXiv preprint arXiv:
  2203.05479,  (2022).

\bibitem{gong2011intrerface}
{\sc J.~{Gong} and J.~{Nordstr\"om}}, {\em {Interface procedures for finite
  difference approximations of the advection-diffusion equation}}, {J. Comput.
  Appl. Math.}, 236 (2011), pp.~602--620,
  \url{https://doi.org/10.1016/j.cam.2011.08.009}.

\bibitem{grant2014cvx}
{\sc M.~Grant and S.~Boyd}, {\em {CVX}: Matlab software for disciplined convex
  programming}, 2014.
\newblock Version 2.2.

\bibitem{gustafsson1995time}
{\sc B.~Gustafsson, H.-O. Kreiss, and J.~Oliger}, {\em Time dependent problems
  and difference methods}, vol.~24, John Wiley \& Sons, 1995.

\bibitem{hesthaven2020rbf}
{\sc J.~S. {Hesthaven}, F.~{M\"onkeberg}, and S.~{Zaninelli}}, {\em {RBF based
  CWENO method}}, in Spectral and high order methods for partial differential
  equations, ICOSAHOM 2018. Selected papers from the ICOSAHOM conference,
  London, UK, July 9--13, 2018, Cham: Springer, 2020, pp.~191--201,
  \url{https://doi.org/10.1007/978-3-030-39647-3_14}.

\bibitem{iske2020ten}
{\sc A.~Iske}, {\em Ten good reasons for using polyharmonic spline
  reconstruction in particle fluid flow simulations}, Continuum Mechanics,
  Applied Mathematics and Scientific Computing: Godunov’s Legacy,  (2020),
  pp.~193--199.

\bibitem{kansa1990multiquadrics}
{\sc E.~J. {Kansa}}, {\em {Multiquadrics -- a scattered data approximation
  scheme with applications to computational fluid-dynamics. II: Solutions to
  parabolic, hyperbolic and elliptic partial differential equations}}, {Comput.
  Math. Appl.}, 19 (1990), pp.~147--161,
  \url{https://doi.org/10.1016/0898-1221(90)90271-K}.

\bibitem{kitson2003skew}
{\sc A.~Kitson, R.~I. McLachlan, and N.~Robidoux}, {\em Skew-adjoint finite
  difference methods on nonuniform grids}, New Zealand J. Math, 32 (2003),
  pp.~139--159.

\bibitem{lazzaro2002radial}
{\sc D.~{Lazzaro} and L.~B. {Montefusco}}, {\em {Radial basis functions for the
  multivariate interpolation of large scattered data sets}}, {J. Comput. Appl.
  Math.}, 140 (2002), pp.~521--536,
  \url{https://doi.org/10.1016/S0377-0427(01)00485-X}.

\bibitem{linders2020properties}
{\sc V.~Linders, J.~Nordstr{\"o}m, and S.~H. Frankel}, {\em Properties of
  {R}unge-{K}utta-summation-by-parts methods}, Journal of Computational
  Physics, 419 (2020), p.~109684.

\bibitem{nordstrom2006conservative}
{\sc J.~{Nordstr\"om}}, {\em {Conservative finite difference formulations,
  variable coefficients, energy estimates and artificial dissipation}}, {J.
  Sci. Comput.}, 29 (2006), pp.~375--404,
  \url{https://doi.org/10.1007/s10915-005-9013-4}.

\bibitem{nordstrom2017roadmap}
{\sc J.~{Nordstr\"om}}, {\em {A roadmap to well posed and stable problems in
  computational physics}}, {J. Sci. Comput.}, 71 (2017), pp.~365--385,
  \url{https://doi.org/10.1007/s10915-016-0303-9}.

\bibitem{offner2019error}
{\sc P.~{\"Offner} and H.~{Ranocha}}, {\em {Error boundedness of discontinuous
  Galerkin methods with variable coefficients}}, {J. Sci. Comput.}, 79 (2019),
  pp.~1572--1607, \url{https://doi.org/10.1007/s10915-018-00902-1}.

\bibitem{pettersson2008improved}
{\sc U.~{Pettersson}, E.~{Larsson}, G.~{Marcusson}, and J.~{Persson}}, {\em
  {Improved radial basis function methods for multi-dimensional option
  pricing}}, {J. Comput. Appl. Math.}, 222 (2008), pp.~82--93,
  \url{https://doi.org/10.1016/j.cam.2007.10.038}.

\bibitem{platte2006eigenvalue}
{\sc R.~B. {Platte} and T.~A. {Driscoll}}, {\em {Eigenvalue stability of radial
  basis function discretizations for time-dependent problems}}, {Comput. Math.
  Appl.}, 51 (2006), pp.~1251--1268,
  \url{https://doi.org/10.1016/j.camwa.2006.04.007}.

\bibitem{ranocha2017extended}
{\sc H.~{Ranocha}, P.~{\"Offner}, and T.~{Sonar}}, {\em {Extended
  skew-symmetric form for summation-by-parts operators and varying Jacobians}},
  {J. Comput. Phys.}, 342 (2017), pp.~13--28,
  \url{https://doi.org/10.1016/j.jcp.2017.04.044}.

\bibitem{shu1988total}
{\sc C.~{Shu}}, {\em {Total-variation-diminishing time discretizations}}, {SIAM
  J. Sci. Stat. Comput.}, 9 (1988), pp.~1073--1084,
  \url{https://doi.org/10.1137/0909073}.

\bibitem{svard2004coordinate}
{\sc M.~{Sv\"ard}}, {\em {On coordinate transformations for summation-by-parts
  operators}}, {J. Sci. Comput.}, 20 (2004), pp.~29--42,
  \url{https://doi.org/10.1023/A:1025881528802}.

\bibitem{svard2014review}
{\sc M.~{Sv\"ard} and J.~{Nordstr\"om}}, {\em {Review of summation-by-parts
  schemes for initial-boundary-value problems}}, {J. Comput. Phys.}, 268
  (2014), pp.~17--38, \url{https://doi.org/10.1016/j.jcp.2014.02.031}.

\bibitem{tolstykh2000using}
{\sc A.~I. Tolstykh}, {\em On using {RBF}-based differencing formulas for
  unstructured and mixed structured-unstructured grid calculations}, in
  Proceedings of the 16th IMACS world congress, vol.~228, Lausanne, 2000,
  pp.~4606--4624.

\bibitem{tominec2021residual}
{\sc I.~Tominec and M.~Nazarov}, {\em Residual viscosity stabilized {RBF-FD}
  methods for solving nonlinear conservation laws}, arXiv preprint
  arXiv:2109.07183,  (2021).

\bibitem{tominec2021stability}
{\sc I.~Tominec, M.~Nazarov, and E.~Larsson}, {\em Stability estimates for
  radial basis function methods applied to time-dependent hyperbolic {PDEs}},
  arXiv preprint arXiv:2110.14548,  (2021).

\bibitem{wendland2002fast}
{\sc H.~Wendland}, {\em Fast evaluation of radial basis functions: {M}ethods
  based on partition of unity}, in Approximation Theory X: Wavelets, Splines,
  and Applications, Citeseer, 2002.

\bibitem{wendland2004scattered}
{\sc H.~Wendland}, {\em Scattered Data Approximation}, vol.~17, Cambridge
  University Press, 2004.

\bibitem{yuan2006discontinuous}
{\sc L.~{Yuan} and C.-W. {Shu}}, {\em {Discontinuous Galerkin method based on
  non-polynomial approximation spaces}}, {J. Comput. Phys.}, 218 (2006),
  pp.~295--323, \url{https://doi.org/10.1016/j.jcp.2006.02.013}.

\end{thebibliography}
